\documentclass[10pt]{article}

\usepackage{enumerate,xspace}
\usepackage{amsmath,amssymb,wasysym}
\usepackage[all]{xy}
\usepackage{proof}
\usepackage[svgnames]{xcolor}
\usepackage{tikz}
\usepackage[numbers,sort&compress]{natbib} 

\usepackage{stmaryrd} 
\usepackage{mathtools}
\usepackage{latexsym}
\usepackage{dsfont}
\usepackage{multicol}
\usepackage{cmll}
\usepackage{multirow}
\usepackage{longtable}
\usepackage{graphicx} 

\usepackage{lscape}
\usepackage{array}

\newcommand{\X}{\mathbb{X}}

\renewcommand{\>}{\rangle}
\newcommand{\<}{\langle}

\usepackage{hyperref} 
\hypersetup{
    colorlinks,
    citecolor=red,
    filecolor=red,
    linkcolor=blue,
    urlcolor=red
}

\newtheorem{observation}{Remark}[section]
\newtheorem{lemma}[observation]{Lemma}  
\newtheorem{theorem}[observation]{Theorem}
\newtheorem{definition}[observation]{Definition}
\newtheorem{example}[observation]{Example}

\newtheorem{proposition}[observation]{Proposition} 
\newtheorem{corollary}[observation]{Corollary}

\usepackage{tikz}
\newcommand*\circled[1]{\tikz[baseline=(char.base)]{
            \node[shape=circle,draw,inner sep=2pt] (char) {#1};}}

\makeatletter

\newdimen\w@dth

\def\setw@dth#1#2{\setbox\z@\hbox{\scriptsize $#1$}\w@dth=\wd\z@
\setbox\@ne\hbox{\scriptsize $#2$}\ifnum\w@dth<\wd\@ne \w@dth=\wd\@ne \fi
\advance\w@dth by 1.2em}

\def\t@^#1_#2{\allowbreak\def\n@one{#1}\def\n@two{#2}\mathrel
{\setw@dth{#1}{#2}
\mathop{\hbox to \w@dth{\rightarrowfill}}\limits
\ifx\n@one\empty\else ^{\box\z@}\fi
\ifx\n@two\empty\else _{\box\@ne}\fi}}
\def\t@@^#1{\@ifnextchar_ {\t@^{#1}}{\t@^{#1}_{}}}

\def\t@left^#1_#2{\def\n@one{#1}\def\n@two{#2}\mathrel{\setw@dth{#1}{#2}
\mathop{\hbox to \w@dth{\leftarrowfill}}\limits
\ifx\n@one\empty\else ^{\box\z@}\fi
\ifx\n@two\empty\else _{\box\@ne}\fi}}
\def\t@@left^#1{\@ifnextchar_ {\t@left^{#1}}{\t@left^{#1}_{}}}

\def\two@^#1_#2{\def\n@one{#1}\def\n@two{#2}\mathrel{\setw@dth{#1}{#2}
\mathop{\vcenter{\hbox to \w@dth{\rightarrowfill}\kern-1.7ex
                 \hbox to \w@dth{\rightarrowfill}}%
       }\limits
\ifx\n@one\empty\else ^{\box\z@}\fi
\ifx\n@two\empty\else _{\box\@ne}\fi}}
\def\tw@@^#1{\@ifnextchar_ {\two@^{#1}}{\two@^{#1}_{}}}

\def\tofr@^#1_#2{\def\n@one{#1}\def\n@two{#2}\mathrel{\setw@dth{#1}{#2}
\mathop{\vcenter{\hbox to \w@dth{\rightarrowfill}\kern-1.7ex
                 \hbox to \w@dth{\leftarrowfill}}%
       }\limits
\ifx\n@one\empty\else ^{\box\z@}\fi
\ifx\n@two\empty\else _{\box\@ne}\fi}}
\def\t@fr@^#1{\@ifnextchar_ {\tofr@^{#1}}{\tofr@^{#1}_{}}}

\newdimen\W@dth
\def\setW@dth#1#2{\setbox\z@\hbox{$#1$}\W@dth=\wd\z@
\setbox\@ne\hbox{$#2$}\ifnum\W@dth<\wd\@ne \W@dth=\wd\@ne \fi
\advance\W@dth by 1.2em}

\def\T@^#1_#2{\allowbreak\def\N@one{#1}\def\N@two{#2}\mathrel
{\setW@dth{#1}{#2}
\mathop{\hbox to \W@dth{\rightarrowfill}}\limits
\ifx\N@one\empty\else ^{\box\z@}\fi
\ifx\N@two\empty\else _{\box\@ne}\fi}}
\def\T@@^#1{\@ifnextchar_ {\T@^{#1}}{\T@^{#1}_{}}}

\def\T@left^#1_#2{\def\N@one{#1}\def\N@two{#2}\mathrel{\setW@dth{#1}{#2}
\mathop{\hbox to \W@dth{\leftarrowfill}}\limits
\ifx\N@one\empty\else ^{\box\z@}\fi
\ifx\N@two\empty\else _{\box\@ne}\fi}}
\def\T@@left^#1{\@ifnextchar_ {\T@left^{#1}}{\T@left^{#1}_{}}}

\def\Tofr@^#1_#2{\def\N@one{#1}\def\N@two{#2}\mathrel{\setW@dth{#1}{#2}
\mathop{\vcenter{\hbox to \W@dth{\rightarrowfill}\kern-1.7ex
                 \hbox to \W@dth{\leftarrowfill}}%
       }\limits
\ifx\N@one\empty\else ^{\box\z@}\fi
\ifx\N@two\empty\else _{\box\@ne}\fi}}
\def\T@fr@^#1{\@ifnextchar_ {\Tofr@^{#1}}{\Tofr@^{#1}_{}}}

\def\Two@^#1_#2{\def\N@one{#1}\def\N@two{#2}\mathrel{\setW@dth{#1}{#2}
\mathop{\vcenter{\hbox to \W@dth{\rightarrowfill}\kern-1.7ex
                 \hbox to \W@dth{\rightarrowfill}}%
       }\limits
\ifx\N@one\empty\else ^{\box\z@}\fi
\ifx\N@two\empty\else _{\box\@ne}\fi}}
\def\Tw@@^#1{\@ifnextchar_ {\Two@^{#1}}{\Two@^{#1}_{}}}

\def\to{\@ifnextchar^ {\t@@}{\t@@^{}}}
\def\from{\@ifnextchar^ {\t@@left}{\t@@left^{}}}
\def\tofro{\@ifnextchar^ {\t@fr@}{\t@fr@^{}}}
\def\To{\@ifnextchar^ {\T@@}{\T@@^{}}}
\def\From{\@ifnextchar^ {\T@@left}{\T@@left^{}}}
\def\Two{\@ifnextchar^ {\Tw@@}{\Tw@@^{}}}
\def\Tofro{\@ifnextchar^ {\T@fr@}{\T@fr@^{}}}

\makeatother

\title{A Tangent Category Perspective on \\ Connections in Algebraic Geometry}
\author{G.S.H. Cruttwell\footnote{Partially supported by an NSERC Discovery grant.}, Jean-Simon Pacaud Lemay\footnote{The author is funded by an ARC DECRA (DE230100303).}, and Elias Vandenberg\footnote{Partially supported by an NSERC undergraduate research award}}

\begin{document}
\allowdisplaybreaks

\maketitle

\begin{abstract}
There is an abstract notion of connection in any tangent category.  In this paper, we show that when applied to the tangent category of affine schemes, this recreates the classical notion of a connection on a module (and similarly, in the tangent category of schemes, this recreates the notion of connection on a quasi-coherent sheaf of modules).  By contrast, we also show that in the tangent category of algebras, there are no non-trivial connections.  
\end{abstract}

\tableofcontents
\newpage

\section{Introduction}

Connections are fundamental tools in various kinds of geometry (such as differential geometry, algebraic geometry, non-commutative geometry, etc.), since they enable one to ``connect'' the infinitesimally close fibres of a bundle.  However, there are many different definitions of connections - even within the same setting (e.g. differential geometry). 
 Thus, it is natural to ask how one can relate and compare these different definitions of connections.

One way to approach this issue is via tangent categories \cite{RosickyTangentCats,cockett2014differential}.  A tangent category is a minimal axiomatic setting for working with differential structure. In particular, a tangent category (Sec \ref{sec:tan_cats}) consists of a category $\X$ equipped with an endofunctor $\mathsf{T}: \X \to \X$, where for an object $A$, $\mathsf{T}(A)$ should be thought as the tangent bundle of $A$, along with various natural transformations which capture the fundamental properties of the tangent bundle, such as a ``projection'' $\mathsf{p}: \mathsf{T} \Rightarrow 1_\mathbb{X}$ and a ``zero'' $0: 1_\mathbb{X} \Rightarrow \mathsf{T}$. 

An abstract notion of connections was developed for tangent categories by Cockett and the first named author in \cite{connections}. The starting point is a \textbf{differential bundle} \cite{cockett2018differential}, which is the generalization of a smooth vector bundle (in the category of smooth manifolds) to an arbitrary tangent category. Briefly, a differential bundle (Sec \ref{sec:diff_bun}) consists of a map ${\mathsf{q}: E \to A}$ in the tangent category such that the ``vertical bundle'' $V_\mathsf{q}$, that is, the pullback
    \[ \xymatrix{V_\mathsf{q} \ar[r] \ar[d] & TE \ar[d]^{\mathsf{T}(\mathsf{q})} \\ A \ar[r]_{0_A} & \mathsf{T}(A)} \]
is trivial, i.e., $V_\mathsf{q} \cong E \times_A E$. On such a structure in a tangent category, a \textbf{connection} (Sec \ref{sec:connections}) consists of a pair of maps
\begin{align*}
\mathsf{K}: \mathsf{T}(E) \to E && \mathsf{H}: \mathsf{T}(A) \times_A E \to \mathsf{T}(E)
\end{align*}
satisfying various axioms (Def \ref{def:K}, \ref{def:H}, \& \ref{def:connection}). The map $\mathsf{K}$ is called the \textbf{vertical connection}, and the map $\mathsf{H}$ is called the \textbf{horizontal connection}. It has previously been shown that in the tangent category of smooth manifolds, differential bundles are the same as vector bundles \cite{macadam2021vector}, and this notion of connection replicates the usual notion of a connection on a vector bundle \cite{connections}.  

The aim of this paper is to show that when applied to tangent categories in algebraic geometry, this abstract definition also recreates one of the key notions of connection there. This is not obvious as the definitions of connections found in algebraic geometry are quite different than the definitions of connections described in tangent categories. In algebraic geometry, the standard notion is of a connection on a module \cite{lang2012algebra,mangiarotti2000connections}, or more generally, a quasi-coherent sheaf of modules \cite{katz1970nilpotent}.  In particular, given a commutative ring $R$, a commutative $R$-algebra $A$, and an $A$-module $M$, a \textbf{connection on $M$} (Def \ref{def:module-connection}) consists of an $R$-linear map
    \[ \nabla: M \to \Omega(A) \otimes_A M \]
(where $\Omega(A)$ is the module of Kähler differentials of $A$ over $R$) which is not $A$-linear, but instead satisfies a Leibniz rule, that is, for all $a \in A$ and $m \in M$: 
    \[ \nabla(am) = \mathsf{d}(a) \otimes_A m + a\nabla(m) \]
This definition of a connection naturally extends to a quasi-coherent sheaf of modules of a scheme $A$ over some base scheme $B$.  
How can we connect this module definition of a connection with the abstract version of a connection in a tangent category?  To begin with, we have to find the right tangent category.  For any commutative ring $R$, the category of commutative $R$-algebras is a tangent category where  the tangent bundle is given by dual numbers.  Interestingly, we show that this tangent has \emph{no} non-trivial tangent category connections (Prop. \ref{prop:no_conn_in_algebra}). 

However, this is not quite the right tangent category for algebraic geometry. It is shown in \cite{cockett2014differential} and later developed more fully in \cite{cruttwell2023differential} that for any commutative ring $R$, the category of affine schemes over $R$ (that is, the opposite of the category of commutative $R$-algebras) is also a tangent category, where for a commutative $R$-algebra $A$, $\mathsf{T}(A)$ is the symmetric algebra (over $A$) of the module of Kähler differentials of $A$. A nice coincidence of structure was then demonstrated by the first and second named authors in \cite{cruttwell2023differential}: in this tangent category, differential bundles over a commutative algebra $A$ correspond exactly to modules over $A$. Similarly, in tangent categories of schemes, differential bundles correspond to quasi-coherent sheaves of modules. 
 
With this setup, the question is then whether the tangent category notion of a connection on a differential bundle, when considered in the tangent category of affine schemes, corresponds to the classical notion of a connection on a module.  That this is indeed the case is the main result of this paper (Thm \ref{thm:affine_connection_correspondence}); a similar result holds almost immediately for the category of schemes, where the corresponding notion is of a connection on a quasi-coherent sheaf of modules (Cor \ref{cor:scheme_connections}).  

However, there is more to be said, particularly about curvature. For a tangent category connection $(\mathsf{K},\mathsf{H})$, its \textbf{curvature} (Def \ref{def:curvature-K}) is given by comparing the following two composites:
\begin{align*}
    \mathsf{T}\mathsf{T}(E) \to^{\mathsf{T}(\mathsf{K})} \mathsf{T}(E) \to^{\mathsf{K}} E && \mathsf{T}\mathsf{T}(E) \to^{\mathsf{c}_E}  \mathsf{T}\mathsf{T}(E) \to^{\mathsf{T}(\mathsf{K})} \mathsf{T}(E) \to^{\mathsf{K}} E
\end{align*}
where $\mathsf{c}$ is one of the structural natural transformations of the tangent category (whose naturality represents the symmetry of mixed partial derivatives). On the other hand, for a connection $\nabla$ on a module $M$, its \textbf{curvature} (Def \ref{def:module-curvature}) is given by the map: 
\begin{align*}
\begin{array}[c]{c}
\nabla^2 \end{array}  := \begin{array}[c]{c} \xymatrixrowsep{1pc}\xymatrixcolsep{3pc}\xymatrix{M \ar[r]^-{ \nabla} & \Omega(A) \otimes_A M \ar[r]^-{ 1_{\Omega(A)} \otimes_A \nabla } & \Omega(A) \otimes_A \Omega(A)  \otimes_A M \ar[r]^-{  \omega \otimes_A 1_M} & \Omega^2(A) \otimes_A M}\end{array}
\end{align*}
Again, these two definitions of curvature look quite different. Nevertheless, we show that these two notions are essentially the same; in fact, they only differ by a factor of 2 (Cor \ref{cor:curvatures_match}). Similarly, in a general tangent category, a connection on the tangent bundle has a notion of \textbf{torsion} (Def \ref{def:torsion}). We investigate what that looks like in this algebraic geometry setting (we have not found an analogous notion in the algebraic geometry literature, so, to the best of our knowledge, this may be new) and define the notion of torsion for the module of Kähler differentials of a commutative algebra\footnote{Unfortunately, this is not related to the usual notion of torsion for a module.}. We then show that these two definitions of torsions similarly only differ by a factor of 2 (Cor \ref{cor:torsion-2}). 

There are several results about connections that are true in any tangent category; for example, one can pullback connections and connections are preserved under retractions. In Section \ref{sec:applying_theory}, we discuss some of these results and compare them to what is known in the literature about connections on modules. In particular, combining several tangent category results implies almost immediately that every smooth affine scheme has a connection on its module of Kähler differentials (Cor \ref{cor:smooth_have_connections}). That said, the corresponding result is not true for connections on more general schemes (Ex \ref{ex:no_connection_projective_line}). 

In addition, the literature on connections in algebraic geometry has surprisingly few concrete examples. We describe several in this paper (Sec \ref{sec:examples} \& \ref{sec:sexamples-curvature}), as it is helpful to see explicit examples when comparing how tangent category connections relate to module connections, and when looking at specific examples of curvature.

\section{Connections in Tangent Categories}\label{sec:tan_cats_full}

In this section, to set up notation and terminology, we begin by briefly reviewing tangent categories and differential bundles (for which we use the same notation as in \cite{cruttwell2023differential}) and then provide a more detailed review of connections. We invite the reader to see more detailed introductions to tangent categories in \cite{cockett2014differential, cruttwell2023differential, RosickyTangentCats}, differential bundles in \cite{ching:diff_bundles,cockett2018differential,cruttwell2023differential,macadam2021vector}, and connections in \cite{connections, lucyshyn2017geometric}.

\subsection{Tangent Categories}\label{sec:tan_cats}

Tangent categories formalize the properties of the tangent bundle on smooth manifolds from classical differential geometry. In this paper, we will, in fact, be working in a \emph{Rosický} tangent category, also called a tangent category with \emph{negatives}. The name Rosický tangent category is in honour of Jiri Rosický, who first introduced tangent categories with negatives (though not under that name) in \cite{RosickyTangentCats}. So a \textbf{Rosický tangent structure} on a category $\mathbb{X}$ is a septuple $\mathbb{T} = (\mathsf{T}, \mathsf{p}, +, 0, \ell, \mathsf{c}, -)$ consisting of: 
\begin{enumerate}[(i)]
\item An endofunctor $\mathsf{T}: \mathbb{X} \to  \mathbb{X}$, called the \textbf{tangent bundle functor};
\item A natural transformation $\mathsf{p}_A: \mathsf{T}(A) \to A$, called the \textbf{projection}, such that for each $n\in \mathbb{N}$, the pullback of $n$ copies of $\mathsf{p}_A$ exists, which we denote as $\mathsf{T}_n(A)$; 
\item A natural transformation\footnote{Note that by the universal property of the pullback, we have functors $\mathsf{T}_n: \mathbb{X} \to \mathbb{X}$.} $+_A: \mathsf{T}_2(A) \to \mathsf{T}(A)$, called the \textbf{sum};
\item A natural transformation $0_A: A \to \mathsf{T}(A)$, called the \textbf{zero};
\item A natural transformation $\ell_A: \mathsf{T}(A) \to \mathsf{T}^2(A)$, called the \textbf{vertical lift};
\item A natural transformation $\mathsf{c}_A: \mathsf{T}^2(A) \to \mathsf{T}^2(A)$, called the \textbf{canonical flip};
\item A natural transformation $-_A: \mathsf{T}(A) \to \mathsf{T}(A)$, called the \textbf{negative},
\end{enumerate}
such that the various axioms in \cite[Def 2.3 \& 3.3]{cockett2014differential} hold. Then a \textbf{Rosický tangent category} is a pair $(\mathbb{X}, \mathbb{T})$ consisting of a category $\mathbb{X}$ and a Rosický tangent structure $\mathbb{T}$ on $\mathbb{X}$.

The main intuition one should have for a tangent category is that the endofunctor $\mathsf{T}$ associates every object $A$ to an object $\mathsf{T}(A)$ that ``behaves like a tangent bundle'' for $A$. This behaviour is captured via the natural transformations of the tangent structure, which encode basic properties of the tangent bundle of a smooth manifold from differential geometry, including the natural projection, being a vector bundle, local triviality, linearity of the derivative, symmetry of the mixed partial derivatives, etc. As such, the canonical example of a (Rosický) tangent category is the category of smooth manifolds. In Sec \ref{sec:ALG-tancat}, we will review the tangent category of commutative algebras, whose tangent structure is induced by dual numbers, and in Sec \ref{sec:AFF-tangcat}, we will review the tangent category of affine schemes, whose tangent structure is induced by Kähler differentials. A list of other examples of tangent categories can be found in \cite[Ex 2.2]{cockett2018differential}. 

\subsection{Differential Bundles}\label{sec:diff_bun}

There are many important notions from differential geometry that can be formalized in a tangent category. In particular, \emph{differential bundles} generalize the notion of smooth vector bundles in a tangent category. In a Rosický tangent category\footnote{The definition we provide here is, in fact, that of a ``differential bundle with negatives''. However, as explained in \cite[Sec 2.10]{cruttwell2023differential}, in a Rosický tangent category, the concepts of a differential bundle and a differential bundle with negatives are the same.}, a \textbf{differential bundle} is a septuple $\mathcal{E} = (A, E, \mathsf{q}, \sigma, \mathsf{z}, \lambda, \iota)$ consisting of:
\begin{enumerate}[(i)]
\item Objects $A$ and $E$;
\item A map $\mathsf{q}: E \to A$, called the \textbf{projection}, such that for each $n \in \mathbb{N}$, the pullback of $n$ copies of $\mathsf{q}$ exists, which we denote as $E_n$,
\item A map $\sigma: E_2 \to E$, called the \textbf{sum};
\item A map $\mathsf{z}: A \to E$, called the \textbf{zero};
\item A map $\lambda: E \to \mathsf{T}(E)$, called the \textbf{lift};
\item A map $\iota: E \to E$, called the \textbf{negative}, 
\end{enumerate}
such that the axioms in \cite[Def 2.3]{cockett2018differential} hold. As a shorthand, we will often denote a differential bundle $\mathcal{E} = (A, E, \mathsf{q}, \sigma, \mathsf{z}, \lambda, \iota)$ simply by its projection $\mathsf{q}$ and say that $\mathsf{q}: E \to A$ is a differential bundle.

In \cite[Thm 4.2.7]{macadam2021vector}, MacAdam showed that in the tangent category of smooth manifolds, differential bundles correspond precisely to smooth vector bundles. In Sec \ref{sec:ALG-tancat} and Sec \ref{sec:AFF-tangcat}, we will review how differential bundles in the tangent category of commutative algebras and the tangent category of affine schemes both correspond to modules \cite[Thm 3.14 \& Thm 4.19]{cruttwell2023differential} (however, this correspondence is \emph{covariant} for commutative algebras but \emph{contravariant} for affine schemes). It is also worth noting that in \cite[Prop 4 \& Thm 6]{ching:diff_bundles}, Ching gives a novel characterization of differential bundles which gives alternative ways to prove these results - see \cite[Ex 16, 23 \& 24]{ching:diff_bundles}.

One way to build differential bundles is from the tangent bundle functor. For every object $A$, its tangent bundle is a differential bundle, that is, $(A, \mathsf{T}(A), \mathsf{p}_A, +_A, 0_A, \ell_A, -_A)$ is a differential bundle \cite[Ex 2.4.(ii)]{cockett2018differential}. We can also apply the tangent bundle functor to a differential bundle, that is, if $(A, E, \mathsf{q}, \sigma, \mathsf{z}, \lambda, \iota)$ is a differential bundle, then $\left(\mathsf{T}(\mathsf{q}): \mathsf{T}(E) \to \mathsf{T}(A), \mathsf{T}(\sigma), \mathsf{T}(\zeta), \mathsf{c}_E \circ \mathsf{T}(\lambda), \mathsf{T}(\iota)  \right)$ is also a differential bundle \cite[Lemma2.5]{cockett2018differential}, which we call the \emph{tangent bundle of a differential bundle}. It is important to note that the canonical flip is used to define the lift for the tangent bundle of a differential bundle. 

Now, a morphism between differential bundles is asked to preserve the projection and the lift. So let ${\mathsf{q}: E \to A}$ and $\mathsf{q}^\prime: E^\prime \to A^\prime$ be differential bundles. A \textbf{differential bundle morphism}\footnote{These were referred to as \emph{linear} differential bundle morphisms in \cite[Def 2.3]{cockett2018differential}; however, since these morphisms are the ones of primary importance in this paper, here we simply refer to them as differential bundle morphisms.} $(f,g): \mathsf{q} \to \mathsf{q}^\prime$ is a pair of maps ${f: E \to E^\prime}$ and $g: A \to A^\prime$ such that the following diagram commutes: 
   \begin{equation}\label{diffbunmap}\begin{gathered} 
   \xymatrixcolsep{5pc}\xymatrix{ E \ar[r]^-{f} \ar[d]_-{\mathsf{q}} & E^\prime \ar[d]^-{\mathsf{q}^\prime} &  E \ar[d]_-{\lambda} \ar[r]^-{f}  & E^\prime  \ar[d]^-{\lambda^\prime}    \\
 A \ar[r]_-{g} & A^\prime & \mathsf{T}(E) \ar[r]_-{\mathsf{T}(f)}  & \mathsf{T}(E^\prime)}  \end{gathered}\end{equation} 
One does not need to assume that differential bundle morphisms preserve the sum, zero, or negative since, surprisingly, this follows from preserving the lift \cite[Prop 2.16]{cockett2018differential}. Other properties and constructions of differential bundle morphisms can be found in \cite[Sec 2.5]{cockett2018differential} and \cite[Sec 2.4]{connections}. For a Rosický tangent category $(\mathbb{X}, \mathbb{T})$, we denote by $\mathsf{DBUN}\left(\mathbb{X}, \mathbb{T} \right)$ its category of differential bundles and differential bundle morphisms between them. 

 One can also add and subtract generalized elements of a differential bundle. Indeed, if $\mathsf{q}: E \to A$ is a differential bundle, and $f: X \to E$ and $g: X \to E$ are maps such that $f \mathsf{q} = g \mathsf{q}$, then define ${f +_{\mathsf{q}} g: X \to E}$ and $f -_{\mathsf{q}} g: X \to E$ respectively as the following composites: 
\[ f+_{\mathsf{q}} g :=  \xymatrixcolsep{5pc}\xymatrix{ X \ar[r]^-{\langle f,g \rangle} & E_2 \ar[r]^-{\sigma} &  E} \]
\[ f-_{\mathsf{q}} g :=  \xymatrixcolsep{5pc}\xymatrix{ X \ar[r]^-{\langle f, \iota \circ g \rangle} & E_2 \ar[r]^-{\sigma} &  E} \]
where $\langle-,- \rangle$ is the pairing operator induced by the universal property of the pullback. 


Lastly, there is also a useful operation involving differential bundles called \emph{bracketing}, which is induced from the universal property of the lift. So for a differential bundle $\mathsf{q}: E \to A$, we say that a map ${f: X \to \mathsf{T}(E)}$ satisfies the \textbf{bracketing condition} if the following equality holds: 
\begin{align}
    \mathsf{T}(\mathsf{q}) \circ f  =  0_A \circ \mathsf{q} \circ \mathsf{p}_E \circ f
\end{align}
If $f: X \to \mathsf{T}(E)$ satisfies the bracketing condition, then there exists a unique map $\lbrace f \rbrace: X \to E$ \cite[Lemma 2.10.(ii)]{cockett2018differential}, such that the following equality holds: 
\begin{align}
    f = \mathsf{T}(\sigma) \circ \left \langle \lambda \circ \lbrace f \rbrace, 0_E \circ \mathsf{p}_E \circ f \right \rangle
\end{align}
The map $\lbrace f \rbrace: X \to E$ is called the \textbf{bracketing} of $f$. 

\subsection{Connections}\label{sec:connections}

We now review the notion of a \emph{connection} in a tangent category, which, as the name suggests, generalizes a notion of connection from differential geometry. In a tangent category, a connection is defined on a differential bundle and is a pair consisting of a \emph{vertical} connection and a \emph{horizontal} connection. Connections on tangent bundles are called \emph{affine connections}. Both vertical connections and horizontal connections are defined in terms of differential bundle morphisms. Moreover, for a Rosický tangent category, it turns out that each part completely determines the connection in the sense that for each horizontal connection, there is a unique (effective) vertical connection, which makes it a connection, and vice versa.  

We begin by reviewing vertical connections. We provide the unpacked version of the definition from \cite[Lemma 3.3]{connections}: 

\begin{definition}\label{def:K} A \textbf{vertical connection} \cite[Def 3.2]{connections} on a differentiable bundle $\mathsf{q}: E \to A$ is a map ${\mathsf{K}: \mathsf{T}(E) \to E}$ such that the following diagrams commute: 
 \begin{equation}\label{eq:K}\begin{gathered}
\xymatrixcolsep{5pc}\xymatrix{ E \ar[r]^-{ \lambda } \ar@/_/@{=}[dr]^-{} & \mathsf{T}(E)  \ar@{}[dl]|(0.4){\text{\normalfont \textbf{[K.1]}}} \ar[d]^-{\mathsf{K}} & \mathsf{T}(E) \ar@{}[dr]|-{\text{\normalfont\textbf{[K.2]} }} \ar[r]^-{\mathsf{K}} \ar[d]_-{\mathsf{p}_E} & E  \ar[d]^-{\mathsf{q}}  \\
& E & E \ar[r]_-{\mathsf{q}} & A  \\
\mathsf{T}(E) \ar[r]^-{ \mathsf{K} } \ar@{}[dr]|-{\text{\normalfont\textbf{[K.3]} }} \ar[d]_-{ \ell_E } & E \ar[d]^-{ \lambda } & \mathsf{T}(E) \ar[r]^-{ \mathsf{K} } \ar@{}[ddr]|-{\text{\normalfont\textbf{[K.4]} }} \ar[d]_-{ \mathsf{T}(\lambda) } & E \ar[dd]^-{ \lambda }  \\
\mathsf{T}^2(E) \ar[r]_-{ \mathsf{T}\left( \mathsf{K} \right) }  & \mathsf{T}(E) & \mathsf{T}^2(E)  \ar[d]_-{ \mathsf{c}_E }  & \\
& & \mathsf{T}^2(E)\ar[r]_-{  \mathsf{T}\left( \mathsf{K} \right) }  & \mathsf{T}(E) } 
   \end{gathered}\end  {equation}
\end{definition}

The four diagrams of (\ref{eq:K}) amount to saying that a vertical connection is a retract of the lift (which is \textbf{[K.1]}) and also a differential bundle morphism in two different ways. Explicitly, \textbf{[K.2]} and \textbf{[K.3]} say that $(\mathsf{K}, \mathsf{q}): \mathsf{p}_E \to \mathsf{q}$ is a differential bundles morphism, while \textbf{[K.2]} (rewriting it using the naturality of $\mathsf{p}$ as $\mathsf{q} \circ \mathsf{K} = \mathsf{p}_A \circ \mathsf{T}(\mathsf{q})$) and \textbf{[K.4]} say that $(\mathsf{K}, \mathsf{p}_E): \mathsf{T}(\mathsf{q}) \to \mathsf{q}$ is a differential bundle morphism as well. 

On the other hand, in order to properly define a horizontal connection, we will need to assume our differential bundles also satisfy what Lucyshyn-Wright calls the extra \emph{Basic Condition} \cite[Sec 3.1]{lucyshyn2017geometric}, which is the natural assumption that a differential bundle also has extra existing pullbacks. So a differential bundle $\mathsf{q}: E \to A$ satisfies the \textbf{Basic Condition} if for every $n,m\in \mathbb{N}$, the pullback of $\mathsf{T}_m(A) \to A$ and $E_n \to A$ exists, and for each $k \in \mathbb{N}$, this pullback is preserved by all $\mathsf{T}^k$. In the special case $m=n=1$, we denote this pullback as follows: 
\[ \xymatrixcolsep{5pc}\xymatrix{ \mathsf{T}(A) \times_A E \ar[r]^-{\pi_1} \ar[d]_-{\pi_0} & E \ar[d]^-{\mathsf{p}_A}  \\
\mathsf{T}(A) \ar[r]_-{\mathsf{p}_A} & A  } \]
Moreover, by the assumption that $\mathsf{T}$ preserves this pullback, we may also set that $\mathsf{T}^2(A) \times_{\mathsf{T}(A)} \mathsf{T}(E) \cong \mathsf{T} \left(\mathsf{T}(A) \times_A E  \right)$.

Here is now the unpacked version of the definition of a horizontal connection from \cite[Lemma 4.6]{connections}:

\begin{definition}\label{def:H} A \textbf{horizontal connection} \cite[Def 4.5]{connections} on a differential bundle $\mathsf{q}: E \to A$ which satisfies the Basic Condition is a map $\mathsf{H}: \mathsf{T}(A) \times_A E \to \mathsf{T}(E)$ such that the following diagrams commute: 
 \begin{equation}\label{eq:H}\begin{gathered}
\xymatrixcolsep{1.75pc}\xymatrix{ \mathsf{T}(A) \times_A E \ar@/_/[dr]_-{\pi_0} \ar[r]^-{ \mathsf{H}} & \mathsf{T}(E) \ar[d]^-{\mathsf{T}(\mathsf{q})} \ar@{}[dl]|(0.4){\text{\normalfont \textbf{[H.1]}}} & \mathsf{T}(A) \times_A E \ar@/_/[dr]_-{\pi_1} \ar[r]^-{ \mathsf{H}} & \mathsf{T}(E) \ar[d]^-{\mathsf{p}_E} \ar@{}[dl]|(0.4){\text{\normalfont \textbf{[H.2]}}}  \\
& \mathsf{T}(A) & & E \\
\mathsf{T}(A) \times_A E  \ar[r]^-{ \mathsf{H}} \ar[d]|-{\ell_A \times_{0_A} 0_E} & \mathsf{T}(E)  \ar[dd]^-{\ell_E} & \mathsf{T}(A) \times_A E \ar[d]|-{0_{\mathsf{T}(A)} \times_{0_A} \lambda} \ar[r]^-{ \mathsf{H}} & \mathsf{T}(E)  \ar[d]^-{\mathsf{T}(\lambda)}  \\
\mathsf{T}^2(A) \times_{\mathsf{T}(A)} \mathsf{T}(E) \ar[d]_-{\cong} & \ar@{}[l]|-{\text{\normalfont\textbf{[H.3]} }}  & \mathsf{T}^2(A) \times_{\mathsf{T}(A)} \mathsf{T}(E) \ar[d]_-{\cong} & \mathsf{T}^2(E) \ar[d]^-{\mathsf{c}_E} \ar@{}[l]|-{\text{\normalfont\textbf{[H.4]} }} \\
\mathsf{T} \left(\mathsf{T}(A) \times_A E  \right)  \ar[r]_-{ \mathsf{T}(\mathsf{H})}  & \mathsf{T}^2(E) & \mathsf{T} \left(\mathsf{T}(A) \times_A E  \right) \ar[r]_-{ \mathsf{T}(\mathsf{H})}  & \mathsf{T}^2(E) 
} 
   \end{gathered}\end  {equation}
\end{definition}

The four diagrams of (\ref{eq:H}) precisely say that a horizontal connection is a differential bundle morphism in two different ways. To see this, we note that the pullback along the projection of a differential bundle is again a differential bundle \cite[Lemma2.7]{cockett2018differential}. Therefore, both ${\pi_0: \mathsf{T}(A) \times_A E \to \mathsf{T}(A)}$ and ${\pi_1: \mathsf{T}(A) \times_A E \to E}$ are differential bundles. Then \textbf{[H.1]} and \textbf{[H.3]} say that $(\mathsf{H}, 1_{\mathsf{T}(A)}): \pi_0 \to \mathsf{T}(\mathsf{q})$ is a differential bundle morphism, while \textbf{[H.2]} and \textbf{[H.4]} say that $(\mathsf{H}, 1_E): \pi_1 \to \mathsf{p}_E$ is a differential bundle morphism. Moreover, for any differential bundle $\mathsf{q}: E \to A$ which satisfies the Basic Condition, there is a canonical map ${\mathsf{U}_\mathsf{q}: \mathsf{T}(E) \to \mathsf{T}(A) \times_A E}$ defined as the pairing:
\begin{align}\mathsf{U}_\mathsf{q} \colon = \langle \mathsf{T}(\mathsf{q}), \mathsf{p}_E \rangle
\end{align}
Then \textbf{[H.1]} and \textbf{[H.2]} together say that a horizontal connection is a retract of $\mathsf{U}_\mathsf{q}$, that is, the following diagram commutes: 
\[ \xymatrixcolsep{5pc}\xymatrix{ \mathsf{T}(E) \ar[r]^-{ \mathsf{U}_\mathsf{q} } \ar@/_/@{=}[dr]^-{} &   \mathsf{T}(A) \times_A E \ar@{}[dl]|(0.4){\text{\normalfont \textbf{[H.U]}}} \ar[d]^-{\mathsf{H}}   \\
& \mathsf{T}(E)  }  \]

A connection consists of a vertical connection and a horizontal connection that are compatible with each other in the following sense: 

\begin{definition}\label{def:connection} A \textbf{connection} \cite[Def 5.2]{connections} on a differential bundle $\mathsf{q}: E \to A$ which satisfies the Basic condition is a pair $(\mathsf{K}, \mathsf{H})$ consisting of a vertical connection $\mathsf{K}$ and a horizontal connection $\mathsf{H}$ on $\mathsf{q}$ such that the following equalities hold:
\begin{align}\label{eq:connection}
\mathsf{K} \circ \mathsf{H} = \mathsf{z} \circ \mathsf{q} \circ \pi_1 && \left( \lambda \circ \mathsf{K} +_{\mathsf{T}(\mathsf{q})} 0_E \circ \mathsf{p}_E \right) +_{\mathsf{p}_E} \mathsf{H} \circ \mathsf{U}_\mathsf{q} = 1_{\mathsf{T}(E)}
\end{align}
\end{definition}

Various interesting properties, constructions, and examples of connections can be found in \cite[Sec 5]{connections}. In particular, we can consider connections on tangent bundles. Indeed, note that for every object $A$, its tangent bundle $\mathsf{p}_A: \mathsf{T}(A) \to A$ satisfies the Basic Condition since the pullback of $\mathsf{T}_m(A) \to A$ and $\mathsf{T}_n(A) \to A$ is $\mathsf{T}_{m+n}(A)$.

\begin{definition} An \textbf{affine (vertical/horizontal) connection} \cite[Def 3.2 \& 5.2]{connections} for an object $A$ is a (vertical/horizontal) connection on its tangent bundle ${\mathsf{p}_A: \mathsf{T}(A) \to A}$. 
\end{definition}

Explicitly, an affine vertical connection is of type $\mathsf{K}: \mathsf{T}^2(A) \to \mathsf{T}(A)$, while an affine horizontal connection is of type $\mathsf{H}: \mathsf{T}_2(A) \to \mathsf{T}^2(A)$. 

In the tangent category of smooth manifolds, connections in the tangent category sense correspond to \emph{linear} connections on a smooth vector bundle \cite{macadam2021vector}, which are also called \textbf{Koszul connections} \cite[pg. 227] {spivakVol2}. 

We now review how every connection is completely determined by its vertical connection or its horizontal connection. In fact, for every horizontal connection, there exists a unique vertical connection which together form a connection\footnote{It is important to note that this statement is true in a Rosický tangent category, but not necessarily true in an arbitrary tangent category. However, since in this paper, we only work in a Rosický tangent category, this is not an issue.}. To build a vertical connection from a horizontal connection, we use the bracketing operation. 

\begin{proposition}\label{prop:hor-ver} \cite[Prop 5.12]{connections} Let $\mathsf{H}: \mathsf{T}(A) \times_A E \to \mathsf{T}(E)$ be a horizontal connection on a differential bundle $\mathsf{q}: E \to A$. Define the map ${\mathsf{K}^\flat: \mathsf{T}(E) \to \mathsf{T}(E)}$ as:
\begin{align}
  \mathsf{K}^\flat \colon =  1_{\mathsf{T}(E)} -_{\mathsf{p}_E} \mathsf{H} \circ \mathsf{U}_\mathsf{q}
\end{align}
Then $\mathsf{K}^\flat$ satisfies the bracketing condition, and so define $\mathsf{K}: \mathsf{T}(E) \to E$ as:
\begin{align}
\mathsf{K} = \left \lbrace \mathsf{K}^\flat \right \rbrace
\end{align}
Then ${\mathsf{K}: \mathsf{T}(E) \to E}$ is a vertical connection on $\mathsf{q}$, and it is the unique vertical connection such that $(\mathsf{K}, \mathsf{H})$ is a connection on $\mathsf{q}$. 
\end{proposition}

On the other hand, as was shown by Lucyshyn-Wright in \cite[Thm 7.2]{lucyshyn2017geometric}, characterizing a vertical connection that belongs to a connection amounts to asking that the tangent bundle of a differential bundle be a biproduct of differential bundles \cite[Sec 4]{lucyshyn2017geometric} of two copies of the differential bundle and the tangent bundle of the base object. This biproduct of differential bundles property can be described in terms of a pullback in the base category. A vertical connection satisfying this property is called \emph{effective}. 

\begin{definition} An \textbf{effective vertical connection} \cite[Thm 6.6]{lucyshyn2017geometric} on a differential bundle $\mathsf{q}: E \to A$ is a vertical connection ${\mathsf{K}: \mathsf{T}(E) \to E}$ on $\mathsf{q}$ such that the following diagram is a pullback: 
 \begin{equation}\label{def:K-pullback}\begin{gathered}
\xymatrixcolsep{5pc}\xymatrix{ & \mathsf{T}(E) \ar[dr]^-{ \mathsf{K} }  \ar[dl]_-{ \mathsf{T}(\mathsf{q}) } \ar[d]_-{ \mathsf{p}_E }    \\
\mathsf{T}(A)  \ar[dr]_-{ \mathsf{p}_A }  & E  \ar[d]_-{ \mathsf{q} }  & E \ar[dl]^-{ \mathsf{q}  } \\
& A  } 
   \end{gathered}\end{equation}
\end{definition}

On the one hand, the vertical connection of any connection is always effective: 

\begin{lemma}\cite[Thm 8.1]{lucyshyn2017geometric} If $(\mathsf{K}, \mathsf{H})$ is a connection on a differential bundle $\mathsf{q}: E \to A$ which satisfies the Basic Condition, then $\mathsf{K}: \mathsf{T}(E) \to E$ is an effective vertical connection on $\mathsf{q}$.  
\end{lemma}

On the other hand, for every effective vertical connection, there exists a unique horizontal connection which together form a connection. 

\begin{proposition}\label{prop:ver-hor}\cite[Thm 8.1]{lucyshyn2017geometric} Let $\mathsf{q}: E \to A$ be a differential bundle which satisfies the Basic Condition, and let ${\mathsf{K}: \mathsf{T}(E) \to E}$ be an effective vertical connection on $\mathsf{q}$. Define $\mathsf{H}: \mathsf{T}(A) \times_A E \to \mathsf{T}(E)$ as the unique map which makes the following diagram commute:
 \begin{equation}\label{def:K-pullback-H}\begin{gathered}
\xymatrixcolsep{5pc}\xymatrix{ & \mathsf{T}(A) \times_A E \ar@{-->}[dd]^-{ \mathsf{H} }  \ar[dr]^-{\pi_1 } \ar@/^3pc/[ddd]^-{\pi_1} \ar@/_2pc/[dddl]_-{\pi_0}   \\
& &  E \ar[d]^-{ \mathsf{q} }  \\
& \mathsf{T}(E) \ar[dr]^-{ \mathsf{K} }  \ar[dl]_-{ \mathsf{T}(\mathsf{q}) } \ar[d]_-{ \mathsf{p}_E } & A \ar[d]^-{ \mathsf{z} }    \\
\mathsf{T}(A)  \ar[dr]_-{ \mathsf{p}_A }  & E  \ar[d]_-{ \mathsf{q} }  & E \ar[dl]^-{ \mathsf{q}  } \\
& A   } 
   \end{gathered}\end{equation}
Then $\mathsf{H}: \mathsf{T}(A) \times_A E \to \mathsf{T}(E)$ is a horizontal connection on $\mathsf{q}$ and it is the unique horizontal connection such that $(\mathsf{K}, \mathsf{H})$ is a connection on $\mathsf{q}$. 
\end{proposition}

Therefore a connection can be completely defined either as an effective vertical connection or as a horizontal connection. 

We conclude this section by discussing the notion of the \textbf{curvature} of a connection and the \textbf{torsion} for an affine connection, where each again corresponds to their namesakes from differential geometry. 

Curvature can, in fact, be defined for any vertical connection: 

\begin{definition}\label{def:curvature-K} The \textbf{curvature} \cite[Def 3.20]{connections} of a vertical connection $\mathsf{K}: \mathsf{T}(E) \to E$ on a differential bundle $\mathsf{q}: E \to A$ is the map $\mathsf{C}_\mathsf{K}: \mathsf{T}^2(E) \to E$ defined as follows: 
\begin{align} \label{eq:curvature-K}
\mathsf{C}_\mathsf{K} \colon = \mathsf{K} \circ \mathsf{T}\left( \mathsf{K} \right) \circ \mathsf{c}_E -_{\mathsf{q}} \mathsf{K} \circ \mathsf{T}\left( \mathsf{K} \right) 
\end{align}
A \textbf{flat} vertical connection \cite[Def 3.20]{connections} on a differential bundle $\mathsf{q}: E \to A$ is a vertical connection ${\mathsf{K}: \mathsf{T}(E) \to E}$ on $\mathsf{q}$ whose curvature is zero, that is: 
\begin{align} 
\mathsf{C}_\mathsf{K} \colon = \mathsf{z} \circ \mathsf{q}_A \circ \mathsf{T}(\mathsf{q}) 
\end{align}
(or, equivalently, $\mathsf{K} \circ \mathsf{T}\left( \mathsf{K} \right) \circ \mathsf{c}_E = \mathsf{K} \circ \mathsf{T}(K)$). Similarly, the \textbf{curvature of a connection} \cite[Def 5.2]{connections} is the curvature of its underlying vertical connection, and a \textbf{flat connection} \cite[Def 5.2]{connections} is a connection whose underlying vertical connection is flat. 
\end{definition}

Similarly, torsion can be defined for any affine connection: 

\begin{definition}\label{def:torsion} The \textbf{torsion} \cite[Def 3.24]{connections} of a vertical connection $\mathsf{K}: \mathsf{T}^2(A) \to \mathsf{T}(A)$ on a tangent bundle $\mathsf{p}_A: \mathsf{T}(A) \to A$ is the map $\mathsf{V}_\mathsf{K}: \mathsf{T}^2(A) \to \mathsf{T}(A)$ defined as follows: 
\begin{align} \label{eq:torsion-K}
\mathsf{V}_\mathsf{K} \colon = \mathsf{K} \circ \mathsf{c}_A -_{\mathsf{p}_A} \mathsf{K} 
\end{align}
A \textbf{torsion-free} vertical connection \cite[Def 3.24]{connections} on a tangent bundle $\mathsf{p}_A: \mathsf{T}(A) \to A$ is a vertical connection $\mathsf{K}: \mathsf{T}^2(A) \to \mathsf{T}(A)$ on $\mathsf{p}_A$ whose torsion is zero that is: 
\begin{align} 
\mathsf{V}_\mathsf{K} \colon = 0_A \circ \mathsf{p}_A \circ \mathsf{p}_{\mathsf{T}(A)} 
\end{align}
Similarly, the \textbf{torsion of an affine connection} \cite[Def 5.2]{connections} is the torsion of its underlying vertical connection, and a \textbf{torsion-free affine connection} \cite[Def 5.2]{connections} is an affine connection whose underlying vertical connection is torsion-free. 
\end{definition}

The torsion of an affine connection can also be expressed using the horizontal connection.  

\begin{proposition}\cite[Prop 5.23]{cockett2018differential} \label{prop:torsion-H} For an affine connection $(\mathsf{K}, \mathsf{H})$ of an object $A$, define the map $\mathsf{V}^\flat_\mathsf{K}: \mathsf{T}^2(A) \to \mathsf{T}^2(A)$ as follows:
\begin{align}
    \mathsf{V}^\flat_\mathsf{K} \colon = \mathsf{c}_A \circ \mathsf{H} \circ \mathsf{U}_{\mathsf{p}_A}  -_{\mathsf{p}_A} \mathsf{H} \circ \mathsf{U}_{\mathsf{p}_A} \circ \mathsf{c}_A 
\end{align}
Then $\mathsf{V}^\flat_\mathsf{K}$ satisfies the bracketing condition and the following equality holds: 
\begin{align}\label{eq:torsion-H}
    \mathsf{V}_\mathsf{K} = \lbrace  \mathsf{V}^\flat_\mathsf{K} \rbrace 
\end{align}
    Moreover, $(\mathsf{K}, \mathsf{H})$ is torsion-free if and only if the following equality holds: 
    \begin{align}\label{eq:torsionfree-H}
      \mathsf{c}_A \circ   \mathsf{H} =  \mathsf{H} \circ \tau_A 
    \end{align}
    where $\tau_A: \mathsf{T}_2(A) \to \mathsf{T}_2(A)$ is the natural isomorphism $\tau_A \colon = \langle \pi_1, \pi_0 \rangle$. 
\end{proposition}

\section{No Connections in Algebra}

In this section, we will prove that there are no (non-trivial) connections in the tangent category of commutative algebras. 

\subsection{Tangent Category of Commutative Algebras and their Differential Bundles}\label{sec:ALG-tancat}

Let us quickly review the tangent category of commutative algebras and then explain how differential bundles in this tangent category correspond to modules. For more details, we invite the reader to see \cite[Sec 3]{cruttwell2023differential}.

Let $R$ be a commutative ring, and $R\text{-}\mathsf{CALG}$ be the category of commutative $R$-algebras and $R$-algebra morphisms between them. For a commutative $R$-algebra $A$, we denote its algebra of dual numbers as:
 \[A[\epsilon] := \lbrace a + b \epsilon \vert~ a,b \in A, \epsilon^2 = 0\rbrace\] 
 where $a$ and $b \varepsilon$ will be used respectively as shorthands for $a + 0\varepsilon$ and $0 + b \varepsilon$. Then, using the same notation as in \cite[Sec 3]{cruttwell2023differential}, this induces a Rosický tangent structure \rotatebox[origin=c]{180}{$\mathbb{T}$} on $R\text{-}\mathsf{CALG}$ as follows: 
 \begin{enumerate}[(i)]
\item The tangent bundle functor $\rotatebox[origin=c]{180}{$\mathsf{T}$}: R\text{-}\mathsf{CALG} \to R\text{-}\mathsf{CALG}$ is defined on objects as $\rotatebox[origin=c]{180}{$\mathsf{T}$}(A) := A[\epsilon]$ and sends an algebra morphism $f: A \to B$ to the algebra morphism ${\rotatebox[origin=c]{180}{$\mathsf{T}$}(f): A[\epsilon] \to B[\epsilon]}$  defined as:
 \[\rotatebox[origin=c]{180}{$\mathsf{T}$}(f)(a + b\epsilon) = f(a) + f(b) \epsilon\] 
\item The projection $\mathsf{p}_{A}: A[\varepsilon] \to A$ is defined as: 
\[\mathsf{p}_{A}(a+b \varepsilon) = a \]
where the pullbacks are described by the multi-variable dual numbers:
\[\rotatebox[origin=c]{180}{$\mathsf{T}$}(A) \!=\! A[\varepsilon_1, \hdots, \varepsilon_n] = \lbrace a + b_1 \varepsilon_1 + \hdots + b_n \varepsilon_n \vert~ \forall a,b_i \in A \text{ and } \varepsilon_i \varepsilon_j = 0 \rbrace \]
\item The sum $+_{A}: A[\varepsilon_1, \varepsilon_2] \to A[\varepsilon]$ is defined as: 
\[ +_A( a + b\varepsilon_1 + c \varepsilon_2) = a + (b+c) \varepsilon \]
\item The zero $0_{A}:  A \to A[\varepsilon]$ is defined as: 
\[0_A(a) = a \]
\item The negative $-_{A}: A[\varepsilon] \to A[\varepsilon]$ is defined as: 
\[-_{A}(a + b \varepsilon) = a - b \varepsilon\]
\end{enumerate}
To describe the vertical lift and the canonical flip, we let $\rotatebox[origin=c]{180}{$\mathsf{T}$}\rotatebox[origin=c]{180}{$\mathsf{T}$}(A)$ denote the ring of dual numbers of the ring of dual numbers:
\begin{align*}
 \rotatebox[origin=c]{180}{$\mathsf{T}$}\rotatebox[origin=c]{180}{$\mathsf{T}$}(A) \!=\! A[\varepsilon][\varepsilon^\prime] = \lbrace a + b \varepsilon + c \varepsilon^\prime + d \varepsilon \varepsilon^\prime \vert~ \forall a,b,c,d \in A \text{ and } \varepsilon^2 = {\varepsilon^\prime}^2 = 0 \rbrace
\end{align*}
\begin{enumerate}[{\em (i)}]
\setcounter{enumi}{6}
\item The vertical lift $\ell_{A}: A[\varepsilon] \to A[\varepsilon][\varepsilon^\prime]$ is defined as: 
\[\ell_{A}(a + b \varepsilon) = a + b \varepsilon\varepsilon^\prime\]
\item The canonical flip $\mathsf{c}_{A}: A[\varepsilon][\varepsilon^\prime] \to A[\varepsilon][\varepsilon^\prime]$ is defined as:  
\[\mathsf{c}_{A}( a + b \varepsilon + c \varepsilon^\prime + d \varepsilon \varepsilon^\prime ) =  a + c \varepsilon + b \varepsilon^\prime + d \varepsilon \varepsilon^\prime\]
\end{enumerate}
So $(R\text{-}\mathsf{CALG}, \rotatebox[origin=c]{180}{$\mathbb{T}$})$ is a Rosický tangent category \cite[Lemma 3.2]{cruttwell2023differential}. 

The differential bundles in $(\mathsf{CRING}, \rotatebox[origin=c]{180}{$\mathbb{T}$})$ correspond precisely to modules \cite[Thm 3.14]{cruttwell2023differential}. Indeed, for a commutative $R$-algebra $A$ and an $A$-module $M$, define the commutative $R$-algebra $M[\varepsilon]$ as follows:
\[ M[\varepsilon] = \lbrace a + m \varepsilon \vert~ a \in A, m \in M \text{ and } \varepsilon^2 =0 \rbrace \] 
where $a$ and $m \varepsilon$ will be used respectively as shorthand for $a + 0\varepsilon$ and $0 + m \varepsilon$. Then $M[\varepsilon]$ is a differential bundle where the projection $\mathsf{q}_{M}: M[\varepsilon] \to A$ is defined as:
\[\mathsf{q}_{M}(a+m \varepsilon) = a. \]
We won't need the precise form of the sum or zero; however, the lift will be important. To describe the lift, we denote $\rotatebox[origin=c]{180}{$\mathsf{T}$}\left(M[\varepsilon]  \right)$, the ring of dual numbers of $M[\varepsilon]$, as follows: 
\begin{gather*}
\rotatebox[origin=c]{180}{$\mathsf{T}$}\left(M[\varepsilon]  \right) \!=\! M[\varepsilon][\varepsilon^\prime] = \lbrace a + m \varepsilon + b \varepsilon^\prime + n \varepsilon \varepsilon^\prime \vert~ \forall a,b \in A, m,n \in M \text{ and } \varepsilon^2 = {\varepsilon^\prime}^2 = 0 \rbrace
\end{gather*}
Then the lift $\lambda_M: M[\varepsilon] \to M[\varepsilon][\varepsilon^\prime]$ is defined as: 
\[\lambda(a + m \varepsilon) = a + m \varepsilon\varepsilon^\prime\]

Conversely, given a differential bundle $\mathsf{q}: E \to A$ in $(R\text{-}\mathsf{CALG}, \rotatebox[origin=c]{180}{$\mathbb{T}$})$, its associated $A$-module is the kernel of the projection $\mathsf{ker}(\mathsf{q})= \lbrace x \vert~ \mathsf{q}(x) = 0 \rbrace$ \cite[Lemma 3.7]{cruttwell2023differential}. Moreover, these constructions are inverses of each other \cite[Sec 3.8]{cruttwell2023differential}, so differential bundles over $A$ in $(R\text{-}\mathsf{CALG}, \rotatebox[origin=c]{180}{$\mathbb{T}$})$ correspond to modules over $A$. To see this as an equivalence of categories, let $\mathsf{MOD}$ be the category whose objects are pairs $(A,M)$ consisting of a commutative $R$-algebra $A$ and an $A$-module $M$, and whose maps $(f,g): (A,M) \to (B,N)$ are pairs consisting of an $R$-algebra morphism $f: A \to B$ and an $A$-module morphism $g: M \to N$ in the sense that $g(am) = f(a) g(m)$ for all $a \in A$ and $m \in M$. Then there is an equivalence $\mathsf{DBUN}\left (R\text{-}\mathsf{CALG}, \rotatebox[origin=c]{180}{$\mathbb{T}$}\right)\simeq \mathsf{MOD}$ \cite[Thm 3.14]{cruttwell2023differential}.  

\subsection{No Connections}\label{sec:noconnections}

We now show that there are no non-trivial connections in $(R\text{-}\mathsf{CALG}, \rotatebox[origin=c]{180}{$\mathbb{T}$})$ in the sense that the only differential bundles that have a connection are the \emph{trivial bundles}. In particular, we prove that if a module has a vertical connection on its associated differential bundle, then said module must be the zero module. For the remainder of this section, we fix a commutative $R$-algebra $A$ and an $A$-module $M$.

\begin{proposition}\label{prop:no_conn_in_algebra} $\mathsf{q}_{M}: M[\varepsilon] \to A$ has a vertical connection if and only if $M = \lbrace 0 \rbrace$. 
\end{proposition}
\begin{proof} For the $\Leftarrow$ direction, we first note that in an arbitrary Rosický tangent category, for every object $A$, trivially $(A, A, 1_A, 1_A, 1_A, 0_A, 1_A)$ is a differential bundle \cite[Ex 2.4.(i)]{cockett2018differential}, which is called the \emph{trivial bundle} over $A$. Moreover, the projection $\mathsf{p}_A: \mathsf{T}(A) \to A$ is a vertical connection on $1_A$. In $(R\text{-}\mathsf{CALG}, \rotatebox[origin=c]{180}{$\mathbb{T}$})$, the trivial bundle over a commutative $R$-algebra $A$ is the one associated to the zero $A$-module $\mathsf{0}$, since $\mathsf{ker}(1_A) = \lbrace 0 \rbrace$ and so $\mathsf{0}[\varepsilon] \cong A$. Thus $\mathsf{q}_{\mathsf{0}}: \mathsf{0}[\varepsilon] \to A$ has a vertical connection. 

For the $\Rightarrow$ direction, we need to prove that for all $m \in M$, we have that $m=0$. So, suppose that we have a vertical connection on ${\mathsf{q}_{M}: M[\varepsilon] \to A}$, that is, an $R$-algebra morphism $\mathsf{K}: M[\varepsilon][\varepsilon^\prime] \to M[\varepsilon]$ satisfying the four diagrams in (\ref{eq:K}). Now \textbf{[K.1]} tells us that: 
\[ a + m \varepsilon =  \mathsf{K}\left( \lambda_M(a+m\varepsilon) \right) = \mathsf{K}\left(a + m\varepsilon\varepsilon^\prime  \right)  \]
So when setting $m=0$, we get that $\mathsf{K}(a) = a$, while setting $a=0$, gives us that $\mathsf{K}( m\varepsilon\varepsilon^\prime ) = m \varepsilon$. Next observe that since $\mathsf{K}$ is an $R$-algebra morphism, we get $m \varepsilon = \mathsf{K}( m\varepsilon) \mathsf{K}(\varepsilon^\prime)$. Now denote $\mathsf{K}( m\varepsilon) = b + n\varepsilon$ and $\mathsf{K}(\varepsilon^\prime) = c + n^\prime \varepsilon$. Denoting $\rotatebox[origin=c]{180}{$\mathsf{T}$}\rotatebox[origin=c]{180}{$\mathsf{T}$}\left(M[\varepsilon]  \right) = M[\varepsilon][\varepsilon^\prime][\varepsilon^{\prime\prime}]$, \textbf{[K.3]} then tells us that: 
\begin{gather*}
  c + n^\prime \varepsilon \varepsilon^\prime = \lambda_M( c + n^\prime \varepsilon ) = \lambda_M\left( \mathsf{K}\left( \varepsilon^\prime \right) \right) = \rotatebox[origin=c]{180}{$\mathsf{T}$}\left( \mathsf{K} \right)(\ell_{M[\varepsilon]}(\varepsilon^\prime) ) \\
  = \rotatebox[origin=c]{180}{$\mathsf{T}$}\left( \mathsf{K} \right) \left( \varepsilon^\prime\varepsilon^{\prime\prime} \right) = \mathsf{K}(\varepsilon^\prime) \varepsilon^\prime = (b^\prime + n^\prime \varepsilon) \varepsilon^\prime = b^\prime \varepsilon^\prime + n^\prime \varepsilon \varepsilon^\prime   
\end{gather*}
Thus the above calculation tells us that $b^\prime=0$, and so $\mathsf{K}(\varepsilon^\prime) = n^\prime \varepsilon$. Similarly, we compute that: 
\begin{gather*}
     b + n\varepsilon \varepsilon^\prime = \lambda_M( b + n\varepsilon ) = \lambda_M\left( \mathsf{K}\left( m \varepsilon \right) \right) = \rotatebox[origin=c]{180}{$\mathsf{T}$}\left( \mathsf{K} \right)(\ell_{M[\varepsilon]}(m \varepsilon) ) \\
     = \rotatebox[origin=c]{180}{$\mathsf{T}$}\left( \mathsf{K} \right) \left( m\varepsilon\varepsilon^{\prime\prime} \right) = \mathsf{K}(m\varepsilon) \varepsilon^\prime = (b + n\varepsilon) \varepsilon^\prime = b \varepsilon^\prime + n \varepsilon \varepsilon^\prime 
\end{gather*}
So the above calculation tells us that $b=0$, and so $\mathsf{K}(m\varepsilon) = n \varepsilon$. However, $\mathsf{K}(m\varepsilon) = n \varepsilon$ and $\mathsf{K}(\varepsilon^\prime) = n^\prime \varepsilon$ tells us that $\mathsf{K}( m\varepsilon) \mathsf{K}(\varepsilon^\prime) =0$. Therefore, we get that $m \varepsilon = \mathsf{K}( m\varepsilon) \mathsf{K}(\varepsilon^\prime) =0$, so $m \varepsilon =0$. Therefore, $m=0$. So we conclude that $M= \lbrace 0 \rbrace$. 
\end{proof}

The same is true for horizontal connections, since recall that every horizontal connection always has an associated vertical connection, which together form a connection (Prop \ref{prop:hor-ver}). So if the associated differential bundle over a module has a horizontal connection, then said differential bundle must also have a vertical connection, and therefore, by the above result, the starting module must be trivial. So we can conclude that:

\begin{corollary} $\mathsf{q}_{M}: M[\varepsilon] \to A$ has a (horizontal/vertical) connection if and only if $M = \lbrace 0 \rbrace$. 
\end{corollary}

\section{Connections in Algebraic Geometry}

In this section, we show that in the tangent category of affine schemes, the tangent category version of connection corresponds precisely to the algebraic geometry version of connection and that the notions of curvature are essentially the same as well.   

\subsection{Tangent Category of Affine Schemes and their Differential Bundles}\label{sec:AFF-tangcat}

Let us quickly review the tangent category of affine schemes and explain how differential bundles in this tangent category also correspond to modules (though contravariantly in this case). For more details, we invite the reader to see \cite[Sec 4]{cruttwell2023differential}. 

By the category of affine schemes, we mean the opposite category of commutative algebras. So, for a commutative ring $R$, we associate the category of affine schemes over $R$ with $R\text{-}\mathsf{CALG}^{op}$. To describe the Rosický tangent structure on $R\text{-}\mathsf{CALG}^{op}$, we describe it in terms of a ``co-Rosický tangent structure'' on $R\text{-}\mathsf{CALG}^{op}$, which consists of a functor $\mathsf{T}: R\text{-}\mathsf{CALG} \to R\text{-}\mathsf{CALG}$ equipped with natural transformations of dual type from those in the Rosický tangent structure definition. 

Starting with the tangent bundle, for a commutative $R$-algebra $A$, we denote its module of Kähler differentials (over $R$) as $\Omega(A)$. Then define $\mathsf{T}(A)$ as the free symmetric $A$-algebra over $\Omega(A)$: 
\[ \mathsf{T}(A) := \mathsf{S}_A \left( \Omega(A) \right) = A \oplus \Omega(A) \oplus \left(\Omega(A) \otimes^s_A \Omega(A) \right) \oplus \hdots  \]
where $\otimes^s_A$ is the symmetrized tensor product over $A$. In \cite[Def 16.5.12.I]{grothendieck1966elements}, Grothendieck calls $\mathsf{T}(A)$ the ``fibré tangente'' (French for tangent bundle) of $A$, while in \cite[Sec 2.6]{jubin2014tangent}, Jubin calls $\mathsf{T}(A)$ the tangent algebra of $A$. Examples of specific tangent algebras can be found in \cite[Ex 4.2]{cruttwell2023differential}. 

However, it will also be useful to have a more explicit description of $\mathsf{T}(A)$.  $\mathsf{T}(A)$ can be defined as the free $A$-algebra generated by the set $\lbrace \mathsf{d}(a) \vert~ a \in A \rbrace$ modulo the equations:
\begin{align*}
  \mathsf{d}(1) = 0 && \mathsf{d}(a+b) = \mathsf{d}(a) + \mathsf{d}(b) && \mathsf{d}(ab) = a \mathsf{d}(b) + b \mathsf{d}(a)
\end{align*}
which are the same equations that are used to construct the module of Kähler differentials of $A$. So an arbitrary element of $\mathsf{T}(A)$ is a finite sum of monomials of the form $a \mathsf{d}(b_1) \hdots \mathsf{d}(b_n)$. Thus the algebra structure $\mathsf{T}(A)$ amounts to essentially the same as that of polynomial rings. Since $\mathsf{T}(A)$ is generated by $a$ and $\mathsf{d}(a)$, for all $a \in A$, to define an $R$-algebra morphism with domain $\mathsf{T}(A)$, it suffices to define it on generators $a$ and $\mathsf{d}(a)$.  With this in mind, we can describe the tangent structure $\mathbb{T}$ on $R\text{-}\mathsf{CALG}^{op}$ from the point of view of $R\text{-}\mathsf{CALG}$. 

\begin{enumerate}[{\em (i)}]
\item The tangent bundle functor $\mathsf{T}: R\text{-}\mathsf{CALG} \to R\text{-}\mathsf{CALG}$ sends an object $A$ to its tangent algebra $\mathsf{T}(A)$, and for an $R$-algebra morphism $f: A \to B$, $\mathsf{T}(f): \mathsf{T}(A) \to \mathsf{T}(B)$ is defined as on generators as follows: 
\begin{align*}
    \mathsf{T}(f)(a) = f(a) && \mathsf{T}(f)(\mathsf{d}(a)) = \mathsf{d}(f(a))
\end{align*}
\item The projection $\mathsf{p}_{A}: A \to \mathsf{T}(A)$ is defined as: 
\[\mathsf{p}_{A}(a) = a\]
The pushout (so the pullback in $R\text{-}\mathsf{CALG}^{op}$) is given by tensoring $n$ copies of $\mathsf{T}(A)$ over $A$: 
\[ \mathsf{T}_n(A) = \mathsf{T}(A)^{\otimes_A^n} = \mathsf{T}(A) \otimes_A \hdots \otimes_A \mathsf{T}(A) \]
where $\otimes_A$ is the tensor product over $A$ of $A$-modules.
\item The sum $+_{A}: \mathsf{T}(A) \to \mathsf{T}(A) \otimes_A \mathsf{T}(A)$ is defined on generators as:
\begin{align*}
    +_A(a) = a \otimes_A 1 = 1 \otimes_A a && +_A(\mathsf{d}(a)) = \mathsf{d}(a) \otimes_A 1 + 1 \otimes_A \mathsf{d}(a)
\end{align*}
\item The zero $0_{A}: \mathsf{T}(A) \to A$ is defined on generators as: 
\begin{align*}
    0_A(a) = a && 0_A(\mathsf{d}(a)) = 0
\end{align*}
\item The negative $-_{A}: \mathsf{T}(A) \to \mathsf{T}(A)$ is defined on generators as:
\begin{align*}
-_A(a) = a && -_A(\mathsf{d}(a))= - \mathsf{d}(a)
\end{align*}
\end{enumerate}
To describe the vertical lift and the canonical flip, we describe $\mathsf{T}^2(A)$ as the free $A$-algebra generated by the set:
\[\lbrace \mathsf{d}(a) \vert ~ a \in A \rbrace \cup \lbrace \mathsf{d}^\prime(a) \vert ~ a \in A \rbrace \cup \lbrace \mathsf{d}^\prime\mathsf{d}(a) \vert ~ a \in A \rbrace\]
modulo the appropriate relations. 
\begin{enumerate}[{\em (i)}]
\setcounter{enumi}{6}
\item The vertical lift $\ell_{A}: \mathsf{T}^2(A) \to \mathsf{T}(A)$ is defined on generators as: 
\begin{gather*}
   \ell_A(a) =a \qquad  \ell_A(\mathsf{d}(a)) =0 \qquad  \ell_A(\mathsf{d}^\prime(a)) =0 \qquad \ell_A( \mathsf{d}^\prime\mathsf{d}(a)) = \mathsf{d}(a)
\end{gather*}
\item The canonical flip $\mathsf{c}_{A}: \mathsf{T}^2(R) \to \mathsf{T}^2(A)$ is defined on generators as: 
\begin{gather*}
   \mathsf{c}_A(a) =a \qquad  \mathsf{c}_A(\mathsf{d}(a)) =\mathsf{d}^\prime(a) \qquad  \mathsf{c}_A(\mathsf{d}^\prime(a)) =\mathsf{d}(a) \qquad \mathsf{c}_A( \mathsf{d}^\prime\mathsf{d}(a)) = \mathsf{d}^\prime\mathsf{d}(a)
\end{gather*}
\end{enumerate}
So $(R\text{-}\mathsf{CALG}^{op}, \mathbb{T})$ is a Rosický tangent category \cite[Lemma 4.3]{cruttwell2023differential}. 


The differential bundles in $(R\text{-}\mathsf{CALG}^{op}, \mathbb{T})$ correspond precisely again to modules \cite[Thm 4.17]{cruttwell2023differential}. Indeed for a commutative $R$-algebra $A$ and an $A$-module $M$, let $\mathsf{S}_A(M)$ be the free symmetric $A$-algebra over $M$, that is: 
\[ \mathsf{S}_A \left( M \right) = \bigoplus \limits_{n=0}^{\infty} M^{{\otimes^s_A}^n} = A \oplus M \oplus \left( M \otimes^s_A M \right) \oplus \hdots \]
Note that as an $R$-algebra, $\mathsf{S}_A \left( M \right)$ is generated by all $a \in A$ and $m \in M$. Now $\mathsf{S}_A(M)$ is a differential bundle where the structure viewed in $R\text{-}\mathsf{CALG}$ is given as follows: 
\begin{enumerate}[{\em (i)}]
\item The projection $\mathsf{q}_{M}: A \to \mathsf{S}_A(M)$ is defined as:
\[ \mathsf{q}_{M}(a) = a \]
where the pushout $\mathsf{S}_A(M)_n$ (so the pullback in $R\text{-}\mathsf{CALG}^{op}$) is defined by tensoring $n$-copies of $\mathsf{S}_A(M)$ over $A$: 
\[\mathsf{S}_A(M)_n := {\mathsf{S}_A(M)}^{\otimes^n_A} = \mathsf{S}_A(M) \otimes_A \hdots \otimes_A \mathsf{S}_A(M) \]
\item The sum $\sigma_{M}: \mathsf{S}_A(M) \to \mathsf{S}_A(M) \otimes_A \mathsf{S}_A(M)$ is defined on generators as:
\begin{align*}
  \sigma_M(a) = a \otimes_A 1 = 1 \otimes_A a && \sigma_M(m) = m \otimes_A 1 + 1 \otimes_A m  
\end{align*}
\item The zero $\mathsf{z}_{M}: \mathsf{S}_A(M) \to M$ is defined on generators as: 
\begin{align*}
\mathsf{z}_M(a) = a && \mathsf{z}_M(m) = 0
\end{align*}
\item The negative $\iota_M: \mathsf{S}_A(M) \to \mathsf{S}_A(M)$ is defined on generators as:
\begin{align*}
    \iota_M(a) = a && \iota_M(m) = -m
\end{align*}
\end{enumerate}
To describe the lift, note that $\mathsf{T}(\mathsf{S}_A(M))$ as $R$-algebra is generated by $a$, $m$, $\mathsf{d}(a)$, and $\mathsf{d}(m)$ for all $a \in R$ and $m \in M$ (and modulo the appropriate equations). 
\begin{enumerate}[{\em (i)}]
\setcounter{enumi}{6}
\item The lift $\lambda_{M}: \mathsf{T}(\mathsf{S}_A(M)) \to \mathsf{S}_A(M)$ is defined on generators as:
\begin{align*}
    \lambda_M(a) =a && \lambda_M(m) =0 && \lambda_M(\mathsf{d}(a)) =0 && \lambda_M(\mathsf{d}(m)) = m
\end{align*}
\end{enumerate}
With this structure, $\mathsf{q}_{M}: \mathsf{S}_A(M) \to A$ is a differential bundle in $(R\text{-}\mathsf{CALG}^{op}, \mathbb{T})$ \cite[Lemma 4.11]{cruttwell2023differential}.

Conversely, given a differential bundle $\mathsf{q}: E \to A$ in $(R\text{-}\mathsf{CALG}^{op}, \mathbb{T})$, its associated $A$-module is given by the set $\lbrace \lambda(\mathsf{d}(x)) \vert~ \forall x \in E \rbrace$ \cite[Lemma 4.9]{cruttwell2023differential}. Moreover, these constructions are inverses of each other \cite[Sec 4.12]{cruttwell2023differential}, so differential bundles over $A$ in $(R\text{-}\mathsf{CALG}^{op}, \mathbb{T})$ correspond to modules over $A$. This results in an equivalence $\mathsf{DBUN}\left (R\text{-}\mathsf{CALG}^{op}, \mathbb{T}\right) \simeq \mathsf{MOD}^{op}$ \cite[Thm 4.17]{cruttwell2023differential}, or equivalently $\mathsf{DBUN}\left (R\text{-}\mathsf{CALG}^{op}, \mathbb{T}\right)^{op} \simeq \mathsf{MOD}$. 

Lastly, we observe that every differential bundle in $(R\text{-}\mathsf{CALG}^{op}, \mathbb{T})$ satisfies the Basic Condition. Indeed, since $R\text{-}\mathsf{CALG}$ is complete, all pushouts exist, and since $\mathsf{T}$ is a left adjoint (in fact its right adjoint is \rotatebox[origin=c]{180}{$\mathsf{T}$} from Sec \ref{sec:ALG-tancat}), then $\mathsf{T}$ preserves all pushouts. So dually, $R\text{-}\mathsf{CALG}^{op}$ has all pullbacks and $\mathsf{T}$ preserves all pullbacks in $R\text{-}\mathsf{CALG}^{op}$. As such, it follows that all differential bundles in $(R\text{-}\mathsf{CALG}^{op}, \mathbb{T})$ do indeed satisfy the Basic condition. So we may consider (effective vertical/horizontal) connections over any differential bundle in $(R\text{-}\mathsf{CALG}^{op}, \mathbb{T})$. 

\subsection{Tangent Category Connections for Affine Schemes}\label{sec:tan_cat_connections_in_aff}

We will now show how, for affine schemes, connections in the tangent category sense correspond to connections in the algebraic geometry sense. For the remainder of this section, we fix a commutative ring $R$, a commutative $R$-algebra $A$, and an $A$-module $M$.

We begin by reviewing the algebraic geometry notion of connection on a module, which is analogous to the notion of Koszul connection expressed in terms of covariant derivatives from differential geometry. In algebraic geometry, a connection can be defined for any quasicoherent sheaf of modules over a scheme \cite{katz1970nilpotent}. In the case of affine schemes, this translates to a notion of connection for a module over an algebra.  For a more in-depth introduction to connections in this latter setting, we invite the reader to see \cite[Chap XIX, Ex 13]{lang2012algebra} and \cite[Chap 8]{mangiarotti2000connections}. 

\begin{definition}\label{def:module-connection} A \textbf{(module)\footnote{These connections are typically referred to in the literature as just ``connections''; however, given that we are discussing several different notions of connection in this paper, we find it useful to refer to these as \emph{module} connections.} connection} \cite[Def 8.2.1]{mangiarotti2000connections} on $M$ is an $R$-linear morphism 
\[ \nabla: M \to \Omega(A) \otimes_A M \] 
such that for all $a \in A$ and $m \in M$, the following equality holds: 
\begin{align}\label{eq:connection-leibniz}
\nabla(am) = a \nabla(m) + \mathsf{d}(a) \otimes_A m 
\end{align}
\end{definition}

It is important to stress that a non-zero connection is an $R$-linear morphism and \emph{not} an $A$-linear morphism. Indeed, if a connection $\nabla$ were $A$-linear then $\nabla(am) = a\nabla(m)$, which would then imply that $\mathsf{d}(a) \otimes_A m = 0$ for all $a \in A$ and $m \in M$. Thus the only possible case of having an $A$-linear connection is when $\Omega(A) \otimes_A M = 0$ and so $\nabla =0$. 

The goal of this section is to show that module connections on $M$ correspond precisely to tangent category connections on the associated differential bundle $\mathsf{q}_M$ (see Sec \ref{sec:AFF-tangcat}). Recall that a tangent category connection is completely determined by either its horizontal connection or its (effective) vertical connection (Prop \ref{prop:ver-hor}).  It turns out that horizontal connections are the ones which are most closely related to module connections.  Thus, we will show how, from a module connection on $M$, we can build a horizontal connection (and then, from that, get an induced (effective) vertical connection).  We will also show that conversely, given a horizontal connection on the differential bundle $\mathsf{q}_M$, we can naturally extract a module connection on the module $M$. We then show that these constructions are inverses to each other. 

So, let us begin by explicitly describing what a horizontal connection on $\mathsf{q}_M$ in the tangent category $(R\text{-}\mathsf{CALG}^{op}, \mathbb{T})$ consists of.  Recall that pushouts in $R\text{-}\mathsf{CALG}$ (so pullbacks in $R\text{-}\mathsf{CALG}^{op}$) are given by taking the tensor product of algebras over the same algebra. So a horizontal connection on $\mathsf{q}_{M}$ is an $R$-algebra morphism of type 
\[ \mathsf{H}: \mathsf{T}\left( \mathsf{S}_A(M) \right) \to \mathsf{T}(A) \otimes_A \mathsf{S}_A(M) \] such that the dual diagrams of (\ref{eq:H}) commute. Since we will make explicit use of a horizontal connection on $\mathsf{q}_M$ to build a module connection on $M$ and vice-versa, it will be useful to write these dual diagrams out in full. 

For \textbf{[H.1]} and \textbf{[H.2]}, we denote the canonical injections of the pushout as $\iota_0: \mathsf{T}(A) \to \mathsf{T}(A) \otimes_A \mathsf{S}_A(M)$ and $\iota_1: \mathsf{S}_A(M) \to \mathsf{T}(A) \otimes_A \mathsf{S}_A(M)$, which are defined as follows for all $w \in \mathsf{T}(A)$ and $v \in \mathsf{S}_A(M)$: 
\begin{align*}
    \iota_0(w) = w \otimes_A 1 &&  \iota_1(v) = 1 \otimes_A v
\end{align*}
Then the following diagrams commute: 
\[ \xymatrixcolsep{2.5pc}\xymatrix{  \mathsf{T}\left( \mathsf{S}_A(M) \right) \ar@{}[dr]|(0.4){\text{\normalfont \textbf{[H.1]}}}  \ar[r]^-{ \mathsf{H}} & \mathsf{T}(A) \otimes_A \mathsf{S}_A(M) \\
\mathsf{T}(A) \ar[u]^-{\mathsf{T}(\mathsf{q}_M)} \ar@/_/[ur]_-{\iota_0} &  } \]
\[ \xymatrixcolsep{2.5pc}\xymatrix{ \mathsf{T}\left( \mathsf{S}_A(M) \right) \ar@{}[dr]|(0.4){\text{\normalfont \textbf{[H.2]}}}  \ar[r]^-{ \mathsf{H}} & \mathsf{T}(A) \otimes_A \mathsf{S}_A(M) \\
\mathsf{S}_A(M) \ar[u]^-{\mathsf{p}_{\mathsf{S}_A(M)}} \ar@/_/[ur]_-{\iota_1} &  } \]
For \textbf{[H.3]} and \textbf{[H.4]}, the canonical isomorphism $\mathsf{T}\left( \mathsf{T}(A) \otimes_A \mathsf{S}_A(M) \right) \xrightarrow{\cong} \mathsf{T}^2(A) \otimes_{\mathsf{T}(A)} \mathsf{T}\left( \mathsf{S}_A(M) \right)$ essentially applies the Leibniz rule, that is, for all $w \in \mathsf{T}(A)$ and $v \in \mathsf{S}_A(M)$:
\begin{gather*}
    w \otimes_A v \xmapsto{\cong} w \otimes_{\mathsf{T}(A)} v \qquad \mathsf{d}\left(  w \otimes_A v \right) \xmapsto{\cong} \mathsf{d}^\prime(w) \otimes_{\mathsf{T}(A)} v + w \otimes_{\mathsf{T}(A)} \mathsf{d}(v) 
\end{gather*}
Then the following diagrams also commute: 
\[\xymatrixcolsep{2.5pc}\xymatrix{ 
\mathsf{T}^2\left( \mathsf{S}_A(M) \right)  \ar[r]^-{ \mathsf{T}(\mathsf{H}) } \ar[dd]_-{\ell_{\mathsf{S}_A(M)}} & \mathsf{T}\left( \mathsf{T}(A) \otimes_A \mathsf{S}_A(M) \right) \ar[d]^-{\cong} \\
\ar@{}[r]|-{\text{\normalfont\textbf{[H.3]} }} & \mathsf{T}^2(A) \otimes_{\mathsf{T}(A)} \mathsf{T}\left( \mathsf{S}_A(M) \right)  \ar[d]|-{\ell_A \otimes_{0_A} 0_{\mathsf{S}_A(M)}}  \\
\mathsf{T}\left( \mathsf{S}_A(M) \right)  \ar[r]_-{ \mathsf{H}} & \mathsf{T}(A) \otimes_A \mathsf{S}_A(M) } \]
\[\xymatrixcolsep{2.5pc}\xymatrix{ 
\mathsf{T}^2\left( \mathsf{S}_A(M) \right)  \ar[r]^-{ \mathsf{T}(\mathsf{H}) } \ar[d]_-{\mathsf{c}_{\mathsf{S}_A(M)}} & \mathsf{T}\left( \mathsf{T}(A) \otimes_A \mathsf{S}_A(M) \right) \ar[d]^-{\cong} \\
\mathsf{T}^2\left( \mathsf{S}_A(M) \right) \ar[d]_-{\mathsf{T}(\lambda_M)} \ar@{}[r]|-{\text{\normalfont\textbf{[H.4]} }} & \mathsf{T}^2(A) \otimes_{\mathsf{T}(A)} \mathsf{T}\left( \mathsf{S}_A(M) \right)  \ar[d]|-{0_{\mathsf{T}(A)} \otimes_{0_A} \lambda}  \\
\mathsf{T}\left( \mathsf{S}_A(M) \right)  \ar[r]_-{ \mathsf{H}} & \mathsf{T}(A) \otimes_A \mathsf{S}_A(M)
}  \]
Now, abusing notation slightly, we have that $M \subset \mathsf{T}\left( \mathsf{S}_A(M) \right)$ and also that $\Omega(A) \otimes_A M \subset \mathsf{T}(A) \otimes_A \mathsf{S}_A(M)$. So, we clearly see how the type of a module connection on $M$ corresponds nicely to the type of a horizontal connection on $\mathsf{q}_M$. 

On the other hand, a vertical connection on $\mathsf{q}_{M}$ is an $R$-algebra morphism of type $\mathsf{K}: \mathsf{S}_A(M) \to \mathsf{T}\left( \mathsf{S}_A(M) \right)$ such that the dual diagrams of (\ref{eq:K}) commute. Then $\mathsf{K}$ is effective if the following is a pushout diagram: 
\[ \xymatrixcolsep{5pc}\xymatrix{ & A \ar[dr]^-{ \mathsf{q}_M }  \ar[dl]_-{ \mathsf{p}_A  } \ar[d]_-{ \mathsf{q}_M }    \\
\mathsf{T}(A)  \ar[dr]_-{  \mathsf{T}(\mathsf{q}_M) }  & \mathsf{S}_A(M)  \ar[d]_-{  \mathsf{p}_{\mathsf{S}_A(M)} }  & \mathsf{S}_A(M) \ar[dl]^-{ \mathsf{K}   } \\
&  \mathsf{T}\left( \mathsf{S}_A(M) \right)  } \]
So a tangent category connection on $\mathsf{q}_M$ is a pair $(\mathsf{K}, \mathsf{H})$ consisting of a (effective) vertical connection and a horizontal connection satisfying the dual of (\ref{eq:connection}). Recall that by Prop \ref{prop:hor-ver} and Prop \ref{prop:ver-hor}, a connection is completely determined by its horizontal connection or its (effective) vertical connection. 

Now let's explain how, given a module connection $\nabla: M \to \Omega(A) \otimes_A M$, we can build a horizontal connection. Define the $R$-algebra morphism $\mathsf{H}_\nabla: \mathsf{T}\left( \mathsf{S}_A(M) \right) \to \mathsf{T}(A) \otimes_A \mathsf{S}_A(M)$ as follows on generators: 
\begin{gather*}
    \mathsf{H}_\nabla(a) = a \otimes_A 1 = 1 \otimes_A a \qquad \mathsf{H}_\nabla(m) = 1 \otimes_A m \qquad  \mathsf{H}_\nabla(\mathsf{d}(a)) = \mathsf{d}(a) \otimes_A 1 \qquad  \mathsf{H}_\nabla(\mathsf{d}(m)) = \nabla(m)
\end{gather*}

\begin{proposition} $\mathsf{H}_\nabla$ is a horizontal connection on $\mathsf{q}_{M}$. 
\end{proposition}
\begin{proof} We first need to explain why $\mathsf{H}_\nabla$ is a well-defined $R$-algebra morphism. The only aspect that needs to be checked is that $\mathsf{H}_\nabla$ behaves well on $\mathsf{d}(am) = a \mathsf{d}(m) + m \mathsf{d}(a)$. However, this follows from the Leibniz rule of $\nabla$: 
\begin{gather*}
   \mathsf{H}_\nabla(\mathsf{d}(am)) = \nabla(am) = a \nabla(m) + \mathsf{d}(a) \otimes_A m \\
   =  \mathsf{H}_\nabla(a)  \mathsf{H}_\nabla(\mathsf{d}(m)) +  \mathsf{H}_\nabla(m) \mathsf{H}_\nabla(\mathsf{d}(a)) = \mathsf{H}_\nabla\left(a \mathsf{d}(m) + m \mathsf{d}(a) \right)
\end{gather*}

Next, we need to show that $\mathsf{H}_\nabla$ satisfies the four necessary diagrams, and it suffices to check these on generators. Observe that the first two required diagrams are immediate by definition of $\mathsf{H}_\nabla$. Indeed, \textbf{[H.1]} holds since $\mathsf{H}_\nabla(a) = a \otimes_A 1 = \iota_0(a)$ and $ \mathsf{H}_\nabla(\mathsf{d}(a)) = \mathsf{d}(a) \otimes_A 1 = \iota_0(\mathsf{d}(a))$, while \textbf{[H.2]} holds since $\mathsf{H}_\nabla(a) = 1 \otimes_A a = \iota_1(a)$ and $\mathsf{H}_\nabla(m) = 1 \otimes_A m= \iota_1(m)$. 

So it remains to show \textbf{[H.3]} and \textbf{[H.4]}. Now note that $\mathsf{T}^2\left( \mathsf{S}_A(M) \right)$ has eight sorts of generators: $a$, $m$, $\mathsf{d}(a)$, $\mathsf{d}(m)$, $\mathsf{d}^\prime(a)$, $\mathsf{d}^\prime(m)$, $\mathsf{d}^\prime\mathsf{d}(a)$, and $\mathsf{d}^\prime\mathsf{d}(m)$. For the following calculations, we set:
\[\nabla(m) = \sum^n_{i=1} \mathsf{d}(a_i) \otimes_A m_i\]

For \textbf{[H.3]}, starting with the $a$ generator, we compute that: 
\begin{gather*}
    a \xmapsto{ \ell_{\mathsf{S}_A(M)} } a \xmapsto{ \mathsf{H}_\nabla } a \otimes_A 1 
\end{gather*}
\begin{gather*}
    a \xmapsto{ \mathsf{T}(\mathsf{H}) } \mathsf{H}(a) = a \otimes_A 1 \xmapsto{\cong} a \otimes_{\mathsf{T}(A)} 1 \xmapsto{\ell_A \otimes_{0_A} 0_{\mathsf{S}_A(M)}} a \otimes_A 1 
\end{gather*}
For the $m$ generator, we get that: 
\begin{gather*}
    m \xmapsto{ \ell_{\mathsf{S}_A(M)} } m \xmapsto{ \mathsf{H}_\nabla } 1 \otimes_A m 
\end{gather*}
\begin{gather*}
    m \xmapsto{ \mathsf{T}(\mathsf{H}) } \mathsf{H}(m) = 1 \otimes_A m \xmapsto{\cong} 1 \otimes_{\mathsf{T}(A)} m \xmapsto{\ell_A \otimes_{0_A} 0_{\mathsf{S}_A(M)}} 1 \otimes_A m 
\end{gather*}
For the $\mathsf{d}(a)$ generator, we get: 
\begin{gather*}
    \mathsf{d}(a) \xmapsto{ \ell_{\mathsf{S}_A(M)} } 0 \xmapsto{ \mathsf{H}_\nabla } 0
\end{gather*}
\begin{gather*}
    \mathsf{d}(a) \xmapsto{ \mathsf{T}(\mathsf{H}) } \mathsf{H}(\mathsf{d}(a)) = \mathsf{d}(a) \otimes_A 1 \xmapsto{\cong} \mathsf{d}(a) \otimes_{\mathsf{T}(A)} 1 \xmapsto{\ell_A \otimes_{0_A} 0_{\mathsf{S}_A(M)}} 0  
\end{gather*}
For the $\mathsf{d}(m)$ generator, we get that: 
\begin{gather*}
    \mathsf{d}(m) \xmapsto{ \ell_{\mathsf{S}_A(M)} } 0 \xmapsto{ \mathsf{H}_\nabla } 0
\end{gather*}
\begin{gather*}
    \mathsf{d}(m) \xmapsto{ \mathsf{T}(\mathsf{H}) } \mathsf{H}(\mathsf{d}(m)) = \nabla(m) = \sum^n_{i=1} \mathsf{d}(a_i) \otimes_A m_i     \xmapsto{\cong} \sum^n_{i=1} \mathsf{d}(a_i) \otimes_{\mathsf{T}(A)} m_i \xmapsto{\ell_A \otimes_{0_A} 0_{\mathsf{S}_A(M)}} 0  
\end{gather*}
For the $\mathsf{d}^\prime(a)$ generator, we get: 
\begin{gather*}
    \mathsf{d}^\prime(a) \xmapsto{ \ell_{\mathsf{S}_A(M)} } 0 \xmapsto{ \mathsf{H}_\nabla } 0
\end{gather*}
\begin{gather*}
    \mathsf{d}^\prime(a) \xmapsto{ \mathsf{T}(\mathsf{H}) } \mathsf{d}\left(\mathsf{H}(a)\right) = \mathsf{d}(a \otimes_A 1) \xmapsto{\cong} \mathsf{d}(a) \otimes_{\mathsf{T}(A)} 1 \xmapsto{\ell_A \otimes_{0_A} 0_{\mathsf{S}_A(M)}} 0  
\end{gather*}
For the $\mathsf{d}^\prime(m)$generator, we get that:
\begin{gather*}
    \mathsf{d}^\prime(m) \xmapsto{ \ell_{\mathsf{S}_A(M)} } 0 \xmapsto{ \mathsf{H}_\nabla } 0
\end{gather*}
\begin{gather*}
    \mathsf{d}^\prime(m) \xmapsto{ \mathsf{T}(\mathsf{H}) } \mathsf{d}\left(\mathsf{H}(m)\right) = \mathsf{d}(1 \otimes_A m) \xmapsto{\cong} 1 \otimes_{\mathsf{T}(A)} \mathsf{d}(m) \xmapsto{\ell_A \otimes_{0_A} 0_{\mathsf{S}_A(M)}} 0  
\end{gather*}
For the $\mathsf{d}^\prime\mathsf{d}(a)$ generator, we get: 
\begin{gather*}
    \mathsf{d}^\prime\mathsf{d}(a) \xmapsto{ \ell_{\mathsf{S}_A(M)} } \mathsf{d}(a) \xmapsto{ \mathsf{H}_\nabla } \mathsf{d}(a) \otimes_A 1
\end{gather*}
\begin{gather*}
    \mathsf{d}^\prime\mathsf{d}(a) \xmapsto{ \mathsf{T}(\mathsf{H}) } \mathsf{d}\left(\mathsf{H}(\mathsf{d}(a))\right) = \mathsf{d}( \mathsf{d}(a) \otimes_A 1)   \xmapsto{\cong} \mathsf{d}^\prime\mathsf{d}(a) \otimes_{\mathsf{T}(A)} 1 \xmapsto{\ell_A \otimes_{0_A} 0_{\mathsf{S}_A(M)}} \mathsf{d}(a) \otimes_A 1   
\end{gather*}
Lastly, for the $\mathsf{d}^\prime\mathsf{d}(m)$ generator, we compute that: 
\begin{gather*}
    \mathsf{d}^\prime(m) \xmapsto{ \ell_{\mathsf{S}_A(M)} } \mathsf{d}(m) \xmapsto{ \mathsf{H}_\nabla } \nabla(m)
\end{gather*}
\begin{gather*}
    \mathsf{d}^\prime\mathsf{d}(m) \xmapsto{ \mathsf{T}(\mathsf{H}) } \mathsf{d}\left(\mathsf{H}(\mathsf{d}(m))\right) = \mathsf{d}( \nabla(m) ) = \sum^n_{i=1} \mathsf{d}\left( \mathsf{d}(a_i) \otimes_A m_i \right) \\ 
    \xmapsto{\cong} \sum^n_{i=1} \mathsf{d}^\prime \mathsf{d}(a_i) \otimes_{\mathsf{T}(A)} m_i + \sum^n_{i=1}  \mathsf{d}(a_i) \otimes_{\mathsf{T}(A)} \mathsf{d}(m_i)   \xmapsto{\ell_A \otimes_{0_A} 0_{\mathsf{S}_A(M)}} \sum^n_{i=1} \mathsf{d}(a_i) \otimes_A m_i = \nabla(m)    
\end{gather*}
So \textbf{[H.3]} holds. 

For \textbf{[H.4]}, on the $a$ generator we first compute that: 
\begin{gather*}
    a \xmapsto{ \mathsf{c}_{\mathsf{S}_A(M)} } a \xmapsto{ \mathsf{T}(\lambda_M) } \lambda_M(a) = a \xmapsto{ \mathsf{H}_\nabla } a \otimes_A 1
\end{gather*}
\begin{gather*}
    a \xmapsto{ \mathsf{T}(\mathsf{H}) } \mathsf{H}(a) = a \otimes_A 1 \xmapsto{\cong} a \otimes_{\mathsf{T}(A)} 1 \xmapsto{0_{\mathsf{T}(A)} \otimes_{0_A} \lambda_M} a \otimes_A 1 
\end{gather*}
For the $m$ generator, we get that: 
\begin{gather*}
    m \xmapsto{ \mathsf{c}_{\mathsf{S}_A(M)} } m \xmapsto{ \mathsf{T}(\lambda_M) } \lambda_M(m) = 0 \xmapsto{ \mathsf{H}_\nabla } 0 
\end{gather*}
\begin{gather*}
    m \xmapsto{ \mathsf{T}(\mathsf{H}) } \mathsf{H}(m) = 1 \otimes_A m \xmapsto{\cong} 1 \otimes_{\mathsf{T}(A)} m \xmapsto{0_{\mathsf{T}(A)} \otimes_{0_A} \lambda_M} 0 
\end{gather*}
For the $\mathsf{d}(a)$ generator, we get:
\begin{gather*}
    \mathsf{d}(a) \xmapsto{ \mathsf{c}_{\mathsf{S}_A(M)} } \mathsf{d}^\prime(a) \xmapsto{ \mathsf{T}(\lambda_M) } \mathsf{d}(\lambda_M(a)) = \mathsf{d}(a) \xmapsto{ \mathsf{H}_\nabla } \mathsf{d}(a) \otimes_A 1 
\end{gather*}
\begin{gather*}
    \mathsf{d}(a) \xmapsto{ \mathsf{T}(\mathsf{H}) } \mathsf{H}(\mathsf{d}(a)) = \mathsf{d}(a) \otimes_A 1 \xmapsto{\cong} \mathsf{d}(a) \otimes_{\mathsf{T}(A)} 1 \xmapsto{0_{\mathsf{T}(A)} \otimes_{0_A} \lambda_M} \mathsf{d}(a) \otimes_A 1   
\end{gather*}
For $\mathsf{d}^\prime(m)$generator, we get that:
\begin{gather*}
    \mathsf{d}(m) \xmapsto{ \mathsf{c}_{\mathsf{S}_A(M)} } \mathsf{d}^\prime(m) \xmapsto{ \mathsf{T}(\lambda_M) } \mathsf{d}(\lambda_M(m)) = 0 \xmapsto{ \mathsf{H}_\nabla } 0 
\end{gather*}
\begin{gather*}
    \mathsf{d}(m) \xmapsto{ \mathsf{T}(\mathsf{H}) } \mathsf{H}(\mathsf{d}(m)) = \nabla(m) = \sum^n_{i=1} \mathsf{d}(a_i) \otimes_A m_i     \xmapsto{\cong} \sum^n_{i=1} \mathsf{d}(a_i) \otimes_A m_i \xmapsto{0_{\mathsf{T}(A)} \otimes_{0_A} \lambda_M} 0  
\end{gather*}
For the $\mathsf{d}^\prime(a)$ generator, we get: 
\begin{gather*}
    \mathsf{d}^\prime(a) \xmapsto{ \mathsf{c}_{\mathsf{S}_A(M)} } \mathsf{d}(a) \xmapsto{ \mathsf{T}(\lambda_M) } \lambda_M( \mathsf{d}(a) ) = 0 \xmapsto{ \mathsf{H}_\nabla } 0
\end{gather*}
\begin{gather*}
    \mathsf{d}^\prime(a) \xmapsto{ \mathsf{T}(\mathsf{H}) } \mathsf{d}\left(\mathsf{H}(a)\right) = \mathsf{d}(1 \otimes_A a) \xmapsto{\cong} 1 \otimes_{\mathsf{T}(A)} \mathsf{d}(a) \xmapsto{0_{\mathsf{T}(A)} \otimes_{0_A} \lambda_M} 0  
\end{gather*}
For $\mathsf{d}^\prime(m)$generator, we get that:
\begin{gather*}
    \mathsf{d}^\prime(m) \xmapsto{ \mathsf{c}_{\mathsf{S}_A(M)} } \mathsf{d}(m) \xmapsto{ \mathsf{T}(\lambda_M) } \lambda_M( \mathsf{d}(m) ) = m \xmapsto{ \mathsf{H}_\nabla } 1 \otimes_A m 
\end{gather*}
\begin{gather*}
    \mathsf{d}^\prime(m) \xmapsto{ \mathsf{T}(\mathsf{H}) } \mathsf{d}\left(\mathsf{H}(m)\right) = \mathsf{d}(1 \otimes_A m) \xmapsto{\cong} 1 \otimes_{\mathsf{T}(A)} \mathsf{d}(m) \xmapsto{0_{\mathsf{T}(A)} \otimes_{0_A} \lambda_M} 1 \otimes_A m   
\end{gather*}
For the $\mathsf{d}^\prime\mathsf{d}(a)$ generator, we get: 
\begin{gather*}
    \mathsf{d}^\prime\mathsf{d}(a) \xmapsto{ \mathsf{c}_{\mathsf{S}_A(M)} } \mathsf{d}^\prime \mathsf{d}(a) \xmapsto{ \mathsf{T}(\lambda_M) } \mathsf{d}\left( \lambda_M( \mathsf{d}(a) ) \right) = 0 \xmapsto{ \mathsf{H}_\nabla } 0 
\end{gather*}
\begin{gather*}
    \mathsf{d}^\prime\mathsf{d}(a) \xmapsto{ \mathsf{T}(\mathsf{H}) } \mathsf{d}\left(\mathsf{H}(\mathsf{d}(a))\right) = \mathsf{d}( \mathsf{d}(a) \otimes_A 1) \xmapsto{\cong} \mathsf{d}^\prime\mathsf{d}(a) \otimes_{\mathsf{T}(A)} 1 \xmapsto{0_{\mathsf{T}(A)} \otimes_{0_A} \lambda_M} 0  
\end{gather*}
Lastly, for the $\mathsf{d}^\prime\mathsf{d}(m)$ generator, we compute that:
\begin{gather*}
    \mathsf{d}^\prime\mathsf{d}(m) \xmapsto{ \mathsf{c}_{\mathsf{S}_A(M)} } \mathsf{d}^\prime \mathsf{d}(m) \xmapsto{ \mathsf{T}(\lambda_M) } \mathsf{d}\left( \lambda_M( \mathsf{d}(m) ) \right) = \mathsf{d}(m) \xmapsto{ \mathsf{H}_\nabla } \nabla(M) 
\end{gather*}
\begin{gather*}
    \mathsf{d}^\prime\mathsf{d}(m) \xmapsto{ \mathsf{T}(\mathsf{H}) } \mathsf{d}\left(\mathsf{H}(\mathsf{d}(m))\right) = \mathsf{d}( \nabla(m) ) = \sum^n_{i=1} \mathsf{d}\left( \mathsf{d}(a_i) \otimes m_i \right) \\ 
    \xmapsto{\cong}  \sum^n_{i=1} \mathsf{d}^\prime \mathsf{d}(a_i) \otimes_{\mathsf{T}(A)} m_i + \sum^n_{i=1}  \mathsf{d}(a_i) \otimes_{\mathsf{T}(A)} \mathsf{d}(m_i)   \xmapsto{0_{\mathsf{T}(A)} \otimes_{0_A} \lambda_M} \sum^n_{i=1} \mathsf{d}(a_i) \otimes m_i = \nabla(m)    
\end{gather*}
So \textbf{[H.4]} holds.
\end{proof}

Now that we have a horizontal connection let us build the induced effective vertical connection. To do so, we first have to discuss subtraction and the bracketing operation viewed in $R\text{-}\mathsf{CALG}$. For a commutative $R$-algebra $B$ and $R$-algebra morphisms $f: \mathsf{S}_A(M) \to B$ and $g: \mathsf{S}_A(M) \to B$, the condition $f \circ \mathsf{q}_M = g \circ \mathsf{q}_M$ precisely says that $f(a) = g(a)$. In this case, the $R$-algebra morphism $f -_{\mathsf{q}_M} g: \mathsf{S}_A(M) \to B$ is defined on generators as follows: 
\begin{align*}
    (f -_{\mathsf{q}_M} g)(a) = f(a) = g(a) && (f -_{\mathsf{q}_M} g)(m) = f(m) - g(m) 
\end{align*}
On the other hand, an $R$-algebra morphism $h: \mathsf{T}\left( \mathsf{S}_A(M) \right) \to B$ satisfies the bracketing condition if on generators $h(\mathsf{d}(a)) = 0$. Then $\lbrace h \rbrace: \mathsf{S}_A(M) \to B$ is worked out to be defined on generators as follows: 
\begin{align*}
    \lbrace h \rbrace (a) = h(a) && \lbrace h \rbrace(m) = h( \mathsf{d}(m) )
\end{align*}
Now the canonical map $\mathsf{U}_{\mathsf{q}_M}:  \mathsf{T}(A) \otimes_A \mathsf{S}_A(M) \to \mathsf{T}\left( \mathsf{S}_A(M) \right)$ is given by multiplication, that is, for $a,b \in A$ and $m \in M$, we have that:  
\begin{gather*}
   \mathsf{U}_{\mathsf{q}_M}(a \otimes_A b) = a b \qquad \mathsf{U}_{\mathsf{q}_M}(a \otimes_A m) = am \qquad  \mathsf{U}_{\mathsf{q}_M}(\mathsf{d}(a) \otimes_A b) = b \mathsf{d}(a) \qquad \mathsf{U}_{\mathsf{q}_M}(\mathsf{d}(a) \otimes_A m) = m \mathsf{d}(a)
\end{gather*}
Using $\mathsf{U}_{\mathsf{q}_M}$, we can extend our connection $\nabla$ into $\mathsf{T}\left( \mathsf{S}_A(M) \right)$, that is, define $\nabla_\mathsf{K}: M \to \mathsf{T}\left( \mathsf{S}_A(M) \right)$ as $\nabla_\mathsf{K}(m) = \mathsf{U}_{\mathsf{q}_M}(\nabla(m))$. Explicitly:
\begin{align*}
    \nabla(m) = \sum^n_{i=1}  \mathsf{d}(a_i) \otimes m_i &&\Longrightarrow&& \nabla_\mathsf{K}(m) = \sum^n_{i=1}  m_i \mathsf{d}(a_i)
\end{align*}
As such, note that by definition we have that $\mathsf{U}_{\mathsf{q}_M}\left( \mathsf{H}_\nabla(m) \right) = \nabla_\mathsf{K}(m)$. Therefore, the $R$-algebra morphism $\mathsf{K}^\flat_\nabla: \mathsf{T}\left( \mathsf{S}_A(M) \right) \to \mathsf{T}\left( \mathsf{S}_A(M) \right)$, defined as $\mathsf{K}^\flat_\nabla = 1_{\mathsf{T}\left( \mathsf{S}_A(M) \right)} -_{\mathsf{q}_M} \mathsf{U}_{\mathsf{q}_M} \circ  \mathsf{H}_\nabla$, is worked out on generators as: 
\begin{gather*}
    \mathsf{K}^\flat_\nabla(a) = a \qquad \mathsf{K}^\flat_\nabla(m) = m \qquad  \mathsf{K}^\flat_\nabla(\mathsf{d}(a)) = 0 \qquad  \mathsf{K}^\flat_\nabla(\mathsf{d}(m)) = \mathsf{d}(m) - \nabla_\mathsf{K}(m)
\end{gather*}
Finally, applying the bracketing operation results in the $R$-algebra morphism $\mathsf{K}_\nabla: \mathsf{S}_A(M) \to \mathsf{T}(\mathsf{S}_A(M))$ given on generators as follows: 
\begin{align}
\mathsf{K}_\nabla(a) = a && \mathsf{K}_\nabla(m) = \mathsf{d}(m) - \nabla_\mathsf{K}(m)
\end{align}

\begin{corollary} $\mathsf{K}_\nabla$ is an effective vertical connection on $\mathsf{q}_{M}$. 
\end{corollary}

Putting these results together, we have:

\begin{proposition} $(\mathsf{K}_\nabla, \mathsf{H}_\nabla)$ is a tangent category connection on $\mathsf{q}_{M}$.
\end{proposition}

Now, let us go in the other direction. So let $\mathsf{H}: \mathsf{T}\left( \mathsf{S}_A(M) \right) \to \mathsf{T}(A) \otimes_A \mathsf{S}_A(M)$ be a horizontal connection on $\mathsf{q}_{M}$. Since we want to give the inverse of the previous construction, we need to show that $\mathsf{H}(\mathsf{d}(-))$ gives a module connection on $M$. In order for this to be well-typed, we need to explain why for all $m \in M$, $\mathsf{H}(\mathsf{d}(m) ) \in \Omega(A) \otimes_A M$. 

Consider the canonical injections and ``projections": 
\begin{gather*}
    j_{\Omega(A)}: \Omega(A) \to \mathsf{T}(A) \qquad j_M: M \to \mathsf{S}_M(A) \\ 
    p_{\Omega(A)}: \mathsf{T}(A) \to \Omega(A) \qquad p_M: \mathsf{S}_M(A) \to M
\end{gather*}
Now since these injections and projections are $A$-linear, we also get:
\begin{gather*}
    j_{\Omega(A)} \otimes_A j_M:  \Omega(A) \otimes_A  M \to \mathsf{T}(A) \otimes_A \mathsf{S}_A(M) \\
    p_{\Omega(A)} \otimes_A p_M:  \mathsf{T}(A) \otimes_A \mathsf{S}_A(M) \to \Omega(A) \otimes_A  M
\end{gather*}
Note of course that $(p_{\Omega(A)} \otimes_A p_M) \circ (j_{\Omega(A)} \otimes_A j_M) = 1_{\Omega(A) \otimes_A M}$. On the other hand, the idempotent $(j_{\Omega(A)} \otimes_A j_M) \circ (p_{\Omega(A)} \otimes_A p_M)$ picks out the $ \Omega(A) \otimes_A  M$ part in $\mathsf{T}(A) \otimes_A \mathsf{S}_A(M)$. In other words, for $W \in \mathsf{T}(A) \otimes_A \mathsf{S}_A(M)$, if:
\[ \left( (j_{\Omega(A)} \otimes_A j_M) \circ (p_{\Omega(A)} \otimes_A p_M) \right) (W) = W\]
then $W \in \Omega(A) \otimes_A  M$. So we will show that when taking $W= \mathsf{H}(\mathsf{d}(m) )$, the above equality holds. To do so, we first observe that: 

\begin{lemma}\label{lift-jp} The following diagrams commute: 
\[     \xymatrixcolsep{5pc}\xymatrix{ \mathsf{T}(A) \ar[d]_-{p_{\Omega(A)}} \ar[r]^-{\mathsf{d}^\prime} & \mathsf{T}^2(A) \ar[d]^-{\ell_{A}} &  \mathsf{S}_A(M) \ar[d]_-{p_{M}} \ar[r]^-{\mathsf{d}} &  \mathsf{T}\left( \mathsf{S}_A(M) \right) \ar[d]^-{\lambda_M} \\
\Omega(A) \ar[r]_-{j_{\Omega(A)}} & \mathsf{T}(A) & M \ar[r]_-{j_M} & \mathsf{S}_M(A)   } \] 
\end{lemma}
\begin{proof} Note that the diagram on the left is a special case of the diagram on the right. That the diagram on the right commutes is a direct consequence of the equivalence between modules and differential bundles \cite[Sec 4.12]{cruttwell2023differential}. Indeed, the equivalence tells us that $M = \mathsf{im}(\lambda_M \circ \mathsf{d})$ \cite[Lemma 4.13]{cruttwell2023differential}, which is precisely the same as saying that the diagram on the right commutes. 
\end{proof}

\begin{lemma} The following diagram commutes: 
\[     \xymatrixcolsep{3pc}\xymatrix{ M \ar[rr]^-{\mathsf{d}} \ar[d]_-{\mathsf{d}} && \mathsf{T}\left( \mathsf{S}_A(M) \right) \ar[dd]^-{ \mathsf{H}} \\
  \mathsf{T}\left( \mathsf{S}_A(M) \right) \ar[d]_-{ \mathsf{H}} \\
  \mathsf{T}(A) \otimes_A \mathsf{S}_A(M) \ar[dr]_-{p_{\Omega(A)} \otimes_A p_M~~} & &   \mathsf{T}(A) \otimes_A \mathsf{S}_A(M)\\
  & \Omega(A) \otimes_A  M \ar[ur]_-{~~~j_{\Omega(A)} \otimes_A j_M}  } \]
  In other words, $\mathsf{H}(\mathsf{d}(m) ) \in \Omega(A) \otimes_A M$. 
\end{lemma}
\begin{proof} We first observe that the following diagram commutes: 
    \[     \xymatrixcolsep{5pc} \xymatrixrowsep{2.5pc}\xymatrix{ M \ar[rr]^-{\mathsf{d}} \ar[dr]^-{\mathsf{d}^\prime\mathsf{d}} \ar[d]_-{\mathsf{d}} & \ar@{}[d]|-{\text{Def. of $\ell$}} & \mathsf{T}\left( \mathsf{S}_A(M) \right) \ar[dddd]^-{ \mathsf{H}} \\
  \mathsf{T}\left( \mathsf{S}_A(M) \right) \ar@{}[dr]|-{\text{Def. of $\mathsf{T}(-)$}} \ar[r]_-{\mathsf{d}^\prime}  \ar[ddd]_-{ \mathsf{H}} & \mathsf{T}^2\left( \mathsf{S}_A(M) \right) \ar[d]|-{\mathsf{T}(H)} \ar[ur]_-{\ell_{\mathsf{S}_A(M)}}  \\
& \mathsf{T}\left( \mathsf{T}(A) \otimes_A \mathsf{S}_A(M) \right) \ar[d]^-{\cong}  \ar@{}[r]|-{\text{\textbf{[H.3]}}} &   \\
& \ar@{}[l]|(0.65){\text{Def. of $\cong$}}\mathsf{T}^2(A) \otimes_{\mathsf{T}(A)} \mathsf{T}\left( \mathsf{S}_A(M) \right) \ar@{}[dd]|-{(\star)}  \ar[dr]|-{\ell_A \otimes_{0_A} 0_{\mathsf{S}_A(M)}} & \\
  \mathsf{T}(A) \otimes_A \mathsf{S}_A(M) \ar@/^3pc/[uur]^-{\mathsf{d}} \ar@/_1pc/[ur]|-{\mathsf{d}^\prime \otimes_{\mathsf{p}_A} \mathsf{p}_{\mathsf{S}_A(M)} +  \mathsf{p}_{\mathsf{T}(A)} \otimes_{\mathsf{p}_A} \mathsf{d}} \ar[dr]|-{p_{\Omega(A)} \otimes_A 1_{\mathsf{S}_A(M)}}  & &   \mathsf{T}(A) \otimes_A \mathsf{S}_A(M)\\
  & \Omega(A) \otimes_A  \mathsf{S}_A(M) \ar[ur]|-{j_{\Omega(A)} \otimes_A 1_{\mathsf{S}_A(M)}}  } \]
where $(\star)$ commutes by combining Lemma \ref{lift-jp} and that by definition we also have that  $ 0_{\mathsf{S}_A(M)} \circ \mathsf{d} = 0$. Similarly, we also have that the following diagram commutes: 
      \[     \xymatrixcolsep{4pc} \xymatrixrowsep{2.5pc}\xymatrix{ M \ar[rrr]^-{\mathsf{d}} \ar[dr]^-{\mathsf{d}^\prime\mathsf{d}} \ar[d]_-{\mathsf{d}} & \ar@{}[d]|-{\text{Def. of $\ell$}} & & \mathsf{T}\left( \mathsf{S}_A(M) \right) \ar[dddd]^-{ \mathsf{H}} \\
  \mathsf{T}\left( \mathsf{S}_A(M) \right) \ar@{}[dr]|-{\text{Def. of $\mathsf{T}(-)$}} \ar[r]_-{\mathsf{d}^\prime}  \ar[ddd]_-{ \mathsf{H}} & \mathsf{T}^2\left( \mathsf{S}_A(M) \right) \ar[d]|-{\mathsf{T}(H)} \ar[r]_-{\mathsf{c}_{\mathsf{S}_A(M)}} & \mathsf{T}^2\left( \mathsf{S}_A(M) \right)  \ar[ur]_-{\mathsf{T}(\lambda_M)} \\
& \mathsf{T}\left( \mathsf{T}(A) \otimes_A \mathsf{S}_A(M) \right) \ar@{}[dl]^-{\text{Def. of $\cong$}} \ar[d]^-{\cong}  \ar@{}[rr]|-{\text{\textbf{[H.4]}}} & &  \\
& \mathsf{T}^2(A) \otimes_{\mathsf{T}(A)} \mathsf{T}\left( \mathsf{S}_A(M) \right) \ar@{}[dd]|-{(\star)}  \ar[drr]|-{0_{\mathsf{T}(A)} \otimes_{0_A} \lambda_M} & \\
  \mathsf{T}(A) \otimes_A \mathsf{S}_A(M) \ar@/^3pc/[uur]^-{\mathsf{d}} \ar@/_1pc/[ur]|-{\mathsf{d}^\prime \otimes_{\mathsf{p}_A} \mathsf{p}_{\mathsf{S}_A(M)} +  \mathsf{p}_{\mathsf{T}(A)} \otimes_{\mathsf{p}_A} \mathsf{d}} \ar[dr]|-{1_{\mathsf{T}(A)} \otimes_A p_M}  & & &   \mathsf{T}(A) \otimes_A \mathsf{S}_A(M)\\
  & \mathsf{T}(A) \otimes_A  M \ar[urr]|-{1_{\mathsf{T}(A)} \otimes_A j_M}  } \]
where here $(\star)$ commutes by combining Lemma \ref{lift-jp} and that by definition $ 0_{\mathsf{T}(A)} \circ \mathsf{d}^\prime = 0$. 

Then, combining both outer diagrams gives us that:
\[ (j_{\Omega(A)} \otimes_A j_M) \circ (p_{\Omega(A)} \otimes_A p_M) \circ \mathsf{H} \circ \mathsf{d} = \mathsf{H} \circ \mathsf{d}\]
So we conclude that $\mathsf{H}(\mathsf{d}(m) ) \in \Omega(A) \otimes_A M$ as desired. 
\end{proof}

As such, we may properly define $\nabla_{\mathsf{H}}: M \to \Omega(A) \otimes_A M$ as follows:
\begin{align}
    \nabla_{\mathsf{H}}(m) = \mathsf{H}(\mathsf{d}(m) )
\end{align}

\begin{proposition}\label{prop:H-nabla} $\nabla_{\mathsf{H}}: M \to \Omega(A) \otimes_A M$ is a module connection on $M$.
\end{proposition}
\begin{proof} We must check that $\nabla_{\mathsf{H}}$ satisfies the Leibniz rule. First note that \textbf{[H.1]} tells us that on generators: 
\begin{gather*}
    \mathsf{H}(a) = a \otimes_A 1 \qquad \mathsf{H}(\mathsf{d}(a)) = \mathsf{d}(a) \otimes_A 1
    \end{gather*}
While \textbf{[H.2]} tells us that on generators:
    \begin{gather*}
    \mathsf{H}(a) = 1 \otimes_A a \qquad \mathsf{H}(m) = 1 \otimes_A m
    \end{gather*}
Then using these identities, as well as the fact that $\mathsf{d}$ satisfies the Leibniz rule and that $\mathsf{H}$ preserves multiplication, we compute in $ \mathsf{T}(A) \otimes_A \mathsf{S}_A(M)$ that:  
    \begin{gather*}
       \mathsf{H}(\mathsf{d}(am)) = \mathsf{H}\left( a \mathsf{d}(m) + \mathsf{d}(a) m \right) = \mathsf{H}(a) \mathsf{H}(\mathsf{d}(m)) + \mathsf{H}(\mathsf{d}(a)) \mathsf{H}(m) \\
        = (a \otimes_A 1) \mathsf{H}(\mathsf{d}(m)) + (\mathsf{d}(a) \otimes_A 1)(1 \otimes_A m) = a \mathsf{H}(\mathsf{d}(m)) + \mathsf{d}(a) \otimes_A m 
    \end{gather*}
As such it follows that $\nabla_{\mathsf{H}}(am) = a \nabla_{\mathsf{H}}(m) + \mathsf{d}(a) \otimes_A m$, as desired.  
\end{proof}

Lastly, it remains to show that these constructions are inverses of each other, giving us our desired bijective correspondence. 

\begin{theorem}\label{thm:affine_connection_correspondence} There is a bijective correspondence between module connections on $M$ and tangent category connections on $\mathsf{q}_M$. 
\end{theorem} 
\begin{proof} Since tangent category connections are completely determined by their horizontal connection, it suffices to show that the above constructions between module connections on $M$ and horizontal connections on $\mathsf{q}_M$ are inverses to each other, that is, we must show that $\mathsf{H}_{\nabla_{\mathsf{H}}} =  \mathsf{H}$ and $ \nabla_{\mathsf{H}_\nabla} = \nabla$. 

So starting with a module connection $\nabla: M \to \Omega(A) \otimes_A M$, using the definition of $\mathsf{H}_\nabla$ and $\nabla_{\mathsf{H}_\nabla}$, we easily see that: 
    \begin{align*}
        \nabla_{\mathsf{H}_\nabla}(m) = \mathsf{H}_\nabla(\mathsf{d}(m) ) = \nabla(m) 
    \end{align*}
so $ \nabla_{\mathsf{H}_\nabla} = \nabla$. 

On the other hand, starting with a horizontal connection $\mathsf{H}: \mathsf{T}\left( \mathsf{S}_A(M) \right) \to \mathsf{T}(A) \otimes_A \mathsf{S}_A(M)$, recall that in the proof of Prop \ref{prop:H-nabla}, we explained that: 
\begin{gather*}
    \mathsf{H}_\nabla(a) = a \otimes_A 1 \qquad \mathsf{H}_\nabla(\mathsf{d}(a)) = \mathsf{d}(a) \otimes_A 1 \qquad \mathsf{H}_\nabla(m) = 1 \otimes_A m
    \end{gather*}
    Then using these, as well as by definition of $\nabla_{\mathsf{H}}$ and $ \mathsf{H}_{\nabla_{\mathsf{H}}}$, on generators we get that: 
    \begin{gather*}
        \mathsf{H}_{\nabla_{\mathsf{H}}}(a) = a \otimes_A 1 = \mathsf{H}(a) \\
        \mathsf{H}_{\nabla_{\mathsf{H}}}(m) = 1 \otimes_A m = \mathsf{H}(m) \\
        \mathsf{H}_{\nabla_{\mathsf{H}}}(\mathsf{d}(a)) = \mathsf{d}(a) \otimes_A 1 =  \mathsf{H}(\mathsf{d}(a)) \\ 
        \mathsf{H}_{\nabla_{\mathsf{H}}}(\mathsf{d}(m)) = \nabla_{\mathsf{H}}(m) =   \mathsf{H}(\mathsf{d}(m))
    \end{gather*}
 So $\mathsf{H}_{\nabla_{\mathsf{H}}} =  \mathsf{H}$. 
\end{proof}

Recall that an affine connection on $A$ is a connection on its tangent bundle. The module associated to the tangent bundle is $\Omega(A)$, so in other words $\mathsf{p}_A = \mathsf{q}_{\Omega(A)}$. So from the above theorem, we immediately get that: 

\begin{corollary} There is a bijective correspondence between module connections on $\Omega(A)$ and affine connections on $A$.
\end{corollary}

\subsection{Connection Examples}\label{sec:examples}

The literature on connections in algebraic geometry has surprisingly few concrete examples. In this section, we briefly consider a few basic examples. In particular, we will use Ex \ref{ex:connections_on_plane} to concretely illustrate the differences and similarities between tangent category connections and module connections.  We will also use these examples as a point of entry into a later discussion of applying tangent category theory to obtain results about connections on modules. 

As to not overload notation, throughout the examples, we will denote elements of $\Omega(A) \otimes_A M$ simply using $\otimes$ instead of $\otimes_A$. 

\begin{example}\label{ex:connections_on_plane}
For any commutative ring $R$, let $A^2_R$ be the affine plane over $R$; as an object of $R\text{-}\mathsf{CALG}^{op}$, this is represented by the polynomial ring $R[x_1, x_2]$.  A natural choice of module over $A^2_R$ is its module of Kähler differentials $\Omega(A^2_R)$, which is the free $A^2_R$-module generated by $\mathsf{d}(x_1)$ and $\mathsf{d}(x_2)$.  Then, a connection on this module is a map of type:  
\[ \nabla: \Omega(A^2_R) \to \Omega(A^2_R) \otimes_{A^2_R} \Omega(A^2_R) \]
which is entirely determined by where it sends $\mathsf{d}(x_1)$ and $\mathsf{d}(x_2)$. Thus $\nabla$ is determined by a collection of eight polynomials $\Gamma^{i}_{jl}(x_1, x_2)$ ($i, j, l \in \{0,1\}$), which determine how $\mathsf{d}(x_1)$ and $\mathsf{d}(x_2)$ get mapped:  
    
    \[ \mathsf{d}(x_1) \mapsto \Gamma^1_{11} \mathsf{d}(x_1) \otimes \mathsf{d}(x_1) + \Gamma^1_{12} \mathsf{d}(x_1) \otimes \mathsf{d}(x_2) + \Gamma^1_{21} \mathsf{d}(x_2) \otimes \mathsf{d}(x_1) + \Gamma^1_{22} \mathsf{d}(x_2) \otimes \mathsf{d}(x_2) \]
    \[ \mathsf{d}(x_2) \mapsto \Gamma^2_{11} \mathsf{d}(x_1) \otimes \mathsf{d}(x_1) + \Gamma^2_{12} \mathsf{d}(x_1) \otimes \mathsf{d}(x_2) + \Gamma^2_{21} \mathsf{d}(x_2) \otimes \mathsf{d}(x_1) + \Gamma^2_{22} \mathsf{d}(x_2) \otimes \mathsf{d}(x_2) \]
By analogy with differential geometry, one could think of these polynomials as the ``Christoffel symbols'' of the connection. As evinced by the proofs in the previous section, the corresponding horizontal connection is essentially the same data, just written in a slightly different way. So, the corresponding map 
    \[ H_{\nabla}: T^2(A^2_R) \to T(A^2_R) \otimes_{A^2_R} T(A^2_R) \]
is defined on generators as follows: 
    \[ x_i \mapsto x_i (1 \otimes 1) \]
    \[ \mathsf{d}(x_i) \mapsto \mathsf{d}(x_i) \otimes 1 \]
    \[ \mathsf{d}^\prime(x_i) \mapsto 1 \otimes \mathsf{d}(x_i) \]
    \[ \mathsf{d}^\prime\mathsf{d}(x_1) \mapsto \Gamma^1_{11} \mathsf{d}(x_1) \otimes \mathsf{d}(x_1) + \Gamma^1_{12} \mathsf{d}(x_1) \otimes \mathsf{d}(x_2) + \Gamma^1_{21} \mathsf{d}(x_2) \otimes \mathsf{d}(x_1) + \Gamma^1_{22} \mathsf{d}(x_2) \otimes \mathsf{d}(x_2)\]
    \[ \mathsf{d}^\prime\mathsf{d}(x_2) \mapsto \Gamma^2_{11} \mathsf{d}(x_1) \otimes \mathsf{d}(x_1) + \Gamma^2_{12} \mathsf{d}(x_1) \otimes \mathsf{d}(x_2) + \Gamma^2_{21} \mathsf{d}(x_2) \otimes \mathsf{d}(x_1) + \Gamma^2_{22} \mathsf{d}(x_2) \otimes \mathsf{d}(x_2) \]
The corresponding (effective) vertical connection again has essentially the same information, just with all polynomials negated. So, the vertical connection is the map of type: 
    \[ K_{\nabla}: T(A^2_R) \to T^2(A^2_R) \]
defined on generators as follows: 
    \[ x_i \mapsto x_i \]
    \[ \mathsf{d}(x_1) \mapsto \mathsf{d}^\prime\mathsf{d}(x_1) - \left( \Gamma^1_{11} \mathsf{d}^\prime(x_1)\mathsf{d}(x_1) + \Gamma^1_{12} \mathsf{d}^\prime(x_1)\mathsf{d}(x_2) + \Gamma^1_{21} \mathsf{d}^\prime(x_2)\mathsf{d}(x_1) + \Gamma^1_{22} \mathsf{d}^\prime(x_2)\mathsf{d}(x_2) \right) \]
    \[ \mathsf{d}(x_2) \mapsto \mathsf{d}^\prime\mathsf{d}(x_2) - \left( \Gamma^2_{11} \mathsf{d}^\prime(x_1)\mathsf{d}(x_1) + \Gamma^2_{12} \mathsf{d}^\prime(x_1)\mathsf{d}(x_2) + \Gamma^2_{21} \mathsf{d}^\prime(x_2)\mathsf{d}(x_1) + \Gamma^2_{22} \mathsf{d}^\prime(x_2)\mathsf{d}(x_2) \right) \]
Note that there is a canonical choice of connection: simply choose all $\Gamma^{i}_{jk}$ to be $0$.   We will revisit this idea in Sec \ref{sec:applying_theory}, particularly in Cor \ref{cor:smooth_have_connections}.  However, it's worth noting that choosing all $\Gamma^{i}_{jk}$ to be $0$ \textbf{does not} mean that the connection is identically $0$. Indeed, even if by setting all $\Gamma^{i}_{jk}=0$, giving us $\nabla(\mathsf{d}(x_i)) =0$, the Leibniz rule still gives us non-zero terms. For example, 
    \[ \nabla(x_1^3\mathsf{d}(x_1)) = \mathsf{d}(x_1^3) \otimes \mathsf{d}(x_1) + x_1^3\nabla(\mathsf{d}(x_1)) = 3x_1^2 \mathsf{d}(x_1) \otimes \mathsf{d}(x_1) \]
\end{example}

\begin{example}\label{ex:connections_on_affine_n_space}
The above example easily generalizes to the affine $n$-space $A^n_R := R[x_1, \hdots, x_n]$. In this case, a module connection $\nabla$ on its module of Kähler differentials consists of a choice of $n \times n \times n$ polynomials $\Gamma^{i}_{jl}$ with associated connection given by
    \[ \mathsf{d}(x_i) \mapsto \sum_j \sum_l \Gamma^{i}_{jl} \mathsf{d}(x)_j \otimes \mathsf{d}(x)_R \]
\end{example}

\begin{example}\label{ex:affine_circle_connection}
Now consider the ``affine circle'' over a commutative ring $R$, that is, the commutative $R$-algebra $A$ defined as: 
    \[ A := R[x,y]/\<x^2 + y^2 - 1\> \]
Again, we will consider connections on its module of Kähler differentials, which can be described as the module: 
    \[ \Omega(A) = (A\mathsf{d}(x) + A\mathsf{d}(y))/\<2x\mathsf{d}(x) + 2y\mathsf{d}(y)\>\]
As in the previous example, a module connection 
    \[ \nabla: \Omega(A) \to \Omega(A) \otimes_A \Omega(A) \]
is entirely determined by where it sends $\mathsf{d}(x)$ and $\mathsf{d}(y)$.  However, in this case, not all choices are possible because of the quotient relation.  For example, if we choose $\nabla(\mathsf{d}(x)) := 0$ and $\nabla(\mathsf{d}(y)) := 0$, this is not well-defined, since:
    \[ \nabla(2x\mathsf{d}(x) + 2y\mathsf{d}(y)) = 2\mathsf{d}(x) \otimes \mathsf{d}(x) + 2 \mathsf{d}(y) \otimes \mathsf{d}(y) \]
which is not equal to $0$ in $\Omega(A) \otimes_A \Omega(A)$. On the other hand, there is a ``canonical'' choice which does work. Define $\nabla$ on generators as:
    \[ \mathsf{d}(x) \mapsto -x \mathsf{d}(x) \otimes \mathsf{d}(x) - x \mathsf{d}(y) \otimes \mathsf{d}(y) \]
    \[ \mathsf{d}(y) \mapsto -y \mathsf{d}(x) \otimes \mathsf{d}(x) - y \mathsf{d}(y) \otimes \mathsf{d}(y) \]
It is straightforward to check this gives a well-defined map and, hence, a module connection on $\Omega(A)$.  The corresponding horizontal and vertical connections are as in Ex \ref{ex:connections_on_plane}. However, it is not immediately clear where this ``canonical'' connection on the circle comes from. We will revisit this idea in Ex \ref{ex:affine_circle_revisited} and show how it can be derived from a much simpler connection.  
\end{example}

\begin{example}\label{ex:higher_spheres}
The above example easily generalizes to the ``affine $n$-sphere''; that is, the commutative $R$-algebra $S^n$ defined as follows: 
    \[ S^n := R[x_1, x_2, \hdots, x_n]/\<x_1^2 + x_2^2 + \hdots, x_n^2 -1\>\]
There is again a canonical connection on its module of Kähler differentials given by: 
    \[ \mathsf{d}(x_i) \mapsto \sum_{j=1}^n -x_i \mathsf{d}(x)_j \otimes \mathsf{d}(x)_j \]
\end{example}

\begin{example} Here is an example of a connection on the module of Kähler differentials on an elliptic curve. Let $R = \mathbb{Q}$ (the rational numbers), and consider the elliptic curve
    \[ \mathbb{Q}[x,y]/\<y^2 - x^3 - 1\> \]
A connection on its module of Kähler differentials is given by 
    \[ \mathsf{d}(x) \mapsto -2x^2 \mathsf{d}(x) \otimes \mathsf{d}(x) - \frac{2}{3}x \mathsf{d}(y) \otimes \mathsf{d}(y)\]
    \[ \mathsf{d}(y) \mapsto 3xy \mathsf{d}(x) \otimes \mathsf{d}(x) + y \mathsf{d}(y) \otimes \mathsf{d}(y) \]
\end{example}

\begin{example}\label{ex:conn_an} Free modules naturally have a choice of connection. So for a commutative $R$-algebra $A$, consider the free $A$-module $A^n$. First, observe that, 
    \[ \Omega(A) \otimes_A A^n \cong \Omega(A)^n \]
It then follows that the natural map $A^n \to \Omega(A)^n$ given by
    \[ (a_1, a_2, \hdots, a_n) \mapsto (\mathsf{d}(a_1), \mathsf{d}(a_2), \hdots, \mathsf{d}(a_n)) \]
induces a module connection on $A^n$. We will revisit this example from a different perspective in the discussion after Lemma \ref{lemma:pullback_connections}. 
\end{example}

We end this section with an example of a module that has no (module) connections. 

\begin{example}
Let $R$ be any commutative ring of characteristic not equal to $2$, and let $A$ be the ``fat point'' (i.e. the ring of dual numbers of $R$): 
    \[ A := R[x]/\<x^2\> \]
We claim that there are no connections on the module of Kähler differentials $\Omega(A)$. Indeed, in this case, $\Omega(A)$ can be described as: 
    \[ \Omega(A) = A\mathsf{d}(x)/\<2x\mathsf{d}(x)\>. \]
Then, a module connection 
    \[ \nabla: \Omega(A) \to \Omega(A) \otimes_A \Omega(A) \]
must be given by: 
    \[ \nabla(\mathsf{d}(x)) = p(x) \mathsf{d}(x) \otimes \mathsf{d}(x) \]
for some $p(x) \in A$.  However, we must have that $\nabla(2x\mathsf{d}(x)) = 0$, but by the Leibniz rule
    \[ \nabla(2x\mathsf{d}(x)) = 2\mathsf{d}(x) \otimes \mathsf{d}(x) + 2x p(x) \mathsf{d}(x) \otimes \mathsf{d}(x) = 2\mathsf{d}(x) \otimes \mathsf{d}(x) \]
which is not equal to $0$ in $\Omega(A) \otimes_A \Omega(A)$. Thus, no connection exists on this module.  
\end{example}

\subsection{Curvature and Torsion}

There is a natural notion of curvature for a connection in the algebraic context. In this section, we will show the relationship between curvature in this sense and curvature in the tangent category sense. While the two concepts are indeed related, there is a small subtlety that differentiates them slightly in the general setting. That said, in a setting where we can divide by $2$, then the two notions of curvature do indeed correspond. We will also discuss the notion of torsion of a connection on the module of Kähler differentials and explain how it relates to the tangent category notion of torsion in a similar way. 

We begin by reviewing the notion of curvature for a connection on a module, see \cite[Chap XIX, Ex.13.(b)]{lang2012algebra} and \cite[Def 8.2.4]{mangiarotti2000connections} for more details. For this section, we again fix a commutative ring $R$, a commutative $R$-algebra $A$, and an $A$-module $M$. Let $\Omega^2(A)$ be the wedge product over $A$ of $\Omega(A)$ with itself, so $\Omega^2(A) = \Omega(A) \wedge_A \Omega(A)$, and consider the canonical map $\omega: \Omega(A) \otimes_A \Omega(A) \to \Omega^2(A)$ defined as: 
\begin{align}\omega\left( x \mathsf{d}(a) \otimes_A y \mathsf{d}(b) \right) = xy \left( \mathsf{d}(a) \wedge_A \mathsf{d}(b) \right)
\end{align}

\begin{definition}\label{def:module-curvature} The \textbf{curvature} \cite[Def 8.2.4]{mangiarotti2000connections} of a module connection $\nabla: M \to \Omega(A) \otimes_A M$ is the $R$-linear map $\nabla^2: M \to \Omega^2(A) \otimes_A M$ defined as the following composite:
\begin{align}
\begin{array}[c]{c}
\nabla^2 \end{array}  := \begin{array}[c]{c} \xymatrixrowsep{1pc}\xymatrixcolsep{3pc}\xymatrix{M \ar[r]^-{ \nabla} & \Omega(A) \otimes_A M \ar[r]^-{ 1_{\Omega(A)} \otimes_A \nabla } & \Omega(A) \otimes_A \Omega(A)  \otimes_A M \ar[r]^-{  \omega \otimes_A 1_M} & \Omega^2(A) \otimes_A M}\end{array}
\end{align}
\end{definition}

On the other hand, given a module connection $\nabla: M \to \Omega(A) \otimes_A M$, the curvature of the induced tangent category connection $(\mathsf{K}_\nabla, \mathsf{H}_\nabla)$ is the $R$-algebra morphism $\mathsf{C}_{\mathsf{K}_\nabla}:  \mathsf{S}_A(M) \to \mathsf{T}^2(\mathsf{S}_A(M))$ defined as the dual of (\ref{eq:curvature-K}), so:
\[ \mathsf{C}_{\mathsf{K}_\nabla} \colon = \mathsf{c}_{\mathsf{S}_A(M)} \circ \mathsf{T}\left( \mathsf{K}_\nabla \right) \circ \mathsf{K}_\nabla  -_{\mathsf{q}_M} \mathsf{T}\left( \mathsf{K}_\nabla \right) \circ \mathsf{K}_\nabla  \]
Then $(\mathsf{K}_\nabla, \mathsf{H}_\nabla)$ is flat if its curvature is given on the generators by: 
\begin{align*}
   \mathsf{C}_{\mathsf{K}_\nabla}(a) = a &&    \mathsf{C}_{\mathsf{K}_\nabla}(m) = 0 
\end{align*}

We will now explain how the curvature of a tangent category connection on $\mathsf{q}_M$ is completely determined by the curvature of its associated module connection on $M$. To do so, consider the map $\psi: \Omega^2(A) \otimes_A M \to \mathsf{T}^2(\mathsf{S}_A(M))$ defined as: 
\[ \psi\left( (\mathsf{d}(a) \wedge \mathsf{d}(b)) \otimes_A m \right) = m \mathsf{d}(a)\mathsf{d}^\prime(b) - m \mathsf{d}^\prime(a)\mathsf{d}(b)   \]
We will use this map to express $\mathsf{C}_{\mathsf{K}_\nabla}$ in terms of $\nabla^2$. As such, it also follows that if the curvature of a module connection on $M$ is zero, then its associated tangent category connection on $\mathsf{q}_M$ is flat.

\begin{lemma}\label{lemma:curvature1} On generators, $\mathsf{C}_{\mathsf{K}_\nabla}$ satisfies the following: 
\begin{align*}
    \mathsf{C}_{\mathsf{K}_\nabla}(a) = a && \mathsf{C}_{\mathsf{K}_\nabla}(m) = \psi\left( \nabla^2(m) \right) 
\end{align*} 
Moreover, if $\nabla^2 = 0$ then $(\mathsf{K}_\nabla, \mathsf{H}_\nabla)$ is flat. 
\end{lemma}
\begin{proof} For the calculations of this proof, we set: 
\begin{align}\label{nabla-m-mi}
    \nabla(m) = \sum^n_{i=1}\mathsf{d}(a_i)  \otimes_A m_i && \nabla(m_i) = \sum^{n_i}_{j=1} \mathsf{d}(a_{i,j}) \otimes_A m_{i,j}
\end{align}
As such, $\nabla^2(m)$ is worked out to be: 
\begin{gather*}
m \xmapsto{\nabla} \nabla(m) = \sum^n_{i=1} \mathsf{d}(a_i)  \otimes_A m_i \\
\xmapsto{1_{\Omega(A)} \otimes_A \nabla} \sum^n_{i=1} \sum^{n_i}_{j=1}  \mathsf{d}(a_i) \otimes_A  \mathsf{d}(a_{i,j}) \otimes_A m_{i,j} \\
\xmapsto{\omega \otimes_A 1_M} \sum^n_{i=1} \sum^{n_i}_{j=1} \left( \mathsf{d}(a_i) \wedge \mathsf{d}(a_{i,j}) \right) \otimes_A m_{i,j} 
\end{gather*}
Therefore, we have that: 
\begin{gather*}
    \nabla^2(m) = \sum^n_{i=1} \sum^{n_i}_{j=1} \left( \mathsf{d}(a_i) \wedge \mathsf{d}(a_{i,j}) \right) \otimes_A m_{i,j} \\
    \psi\left( \nabla^2(m) \right) = \sum^n_{i=1} \sum^{n_i}_{j=1} m_{i,j} \mathsf{d}(a_i) \mathsf{d}^\prime(a_{i,j}) -  m_{i,j} \mathsf{d}^\prime(a_i) \mathsf{d}(a_{i,j})
\end{gather*}

On the other hand, let's consider the curvature of $\mathsf{K}_\nabla$. To do so, let's work out $\mathsf{T}(\mathsf{K}_\nabla) \circ \mathsf{K}_\nabla$ on generators: 

\begin{gather*}
    a \xmapsto{\mathsf{K}_\nabla} a  \xmapsto{\mathsf{T}(\mathsf{K}_\nabla)} \mathsf{K}_\nabla(a) = a 
\end{gather*}

\begin{gather*}
    m \xmapsto{\mathsf{K}_\nabla} \mathsf{d}(m) - \nabla_\mathsf{K}(m) = \mathsf{d}(m) - \sum^n_{i=1} m_i \mathsf{d}(a_i) \\
    \xmapsto{\mathsf{T}(\mathsf{K}_\nabla)} \mathsf{d}^\prime\left( \mathsf{K}_\nabla(m) \right) - \sum^n_{i=1} \mathsf{K}_\nabla(m_i) \mathsf{d}^\prime\left( \mathsf{K}_\nabla(a_i) \right)  \\
    = \mathsf{d}^\prime\left( \mathsf{d}(m) - \nabla_\mathsf{K}(m) \right) - \sum^n_{i=1} \left( \mathsf{d}(m_i) - \nabla_\mathsf{K}(m_i) \right) \mathsf{d}^\prime(a_i)   \\
    = \mathsf{d}^\prime \mathsf{d}(m) - \sum^n_{i=1} \mathsf{d}^\prime\left( m_i \mathsf{d}(a_i) \right)  - \sum^n_{i=1} \mathsf{d}(m_i) \mathsf{d}^\prime(a_i)  +  \sum^n_{i=1} \sum^{n_i}_{j=1} m_{i,j} \mathsf{d}(a_{i,j})  \mathsf{d}^\prime\left( a_i \right)  \\
 = \underbrace{\mathsf{d}^\prime \mathsf{d}(m)}_{\circled{1}} - \underbrace{\sum^n_{i=1} m_i \mathsf{d}^\prime\mathsf{d}(a_i)}_{\circled{2}} - \underbrace{\sum^n_{i=1}   \mathsf{d}(a_i) \mathsf{d}^\prime(m_i) }_{\circled{3}} - \underbrace{\sum^n_{i=1} \mathsf{d}(m_i) \mathsf{d}^\prime(a_i)}_{\circled{4}} +  \sum^n_{i=1} \sum^{n_i}_{j=1} m_{i,j} \mathsf{d}(a_{i,j})  \mathsf{d}^\prime\left( a_i \right) 
\end{gather*}
Then on generators $\mathsf{c}_{\mathsf{S}_A(M)}\circ \mathsf{T}(\mathsf{K}_\nabla) \circ \mathsf{K}_\nabla$ gives us: 
\begin{gather*}
    a \xmapsto{\mathsf{T}(\mathsf{K}_\nabla)\circ \mathsf{K}_\nabla} a   \xmapsto{\mathsf{c}_{\mathsf{S}_A(M)}} a 
\end{gather*}

\begin{gather*}
    m \xmapsto{\mathsf{T}(\mathsf{K}_\nabla)\circ \mathsf{K}_\nabla} \\
    \mathsf{d}^\prime \mathsf{d}(m) - \sum^n_{i=1} m_i \mathsf{d}^\prime\mathsf{d}(a_i) - \sum^n_{i=1}   \mathsf{d}(a_i) \mathsf{d}^\prime(m_i)  - \sum^n_{i=1} \mathsf{d}(m_i) \mathsf{d}^\prime(a_i)  +  \sum^n_{i=1} \sum^{n_i}_{j=1} m_{i,j} \mathsf{d}(a_{i,j})  \mathsf{d}^\prime\left( a_i \right) \\
    \xmapsto{\mathsf{c}_{\mathsf{S}_A(M)}}\underbrace{\mathsf{d}^\prime \mathsf{d}(m)}_{\circled{1}} - \underbrace{\sum^n_{i=1} m_i \mathsf{d}^\prime\mathsf{d}(a_i)}_{\circled{2}} - \underbrace{\sum^n_{i=1}   \mathsf{d}^\prime(a_i) \mathsf{d}(m_i)}_{\circled{4}}  - \underbrace{\sum^n_{i=1} \mathsf{d}^\prime(m_i) \mathsf{d}(a_i)}_{\circled{3}} +  \sum^n_{i=1} \sum^{n_i}_{j=1} m_{i,j} \mathsf{d}^\prime(a_{i,j})  \mathsf{d}\left( a_i \right)
\end{gather*}

Note the same four terms in $\mathsf{c}_{\mathsf{S}_A(M)}\circ \mathsf{T}(\mathsf{K}_\nabla) \circ \mathsf{K}_\nabla$ and $\mathsf{T}(\mathsf{K}_\nabla) \circ \mathsf{K}_\nabla$. Therefore, their subtraction (in the differential bundle sense) is easily computed out on the $m$ generator to be: 
\begin{gather*}
    \mathsf{C}_{\mathsf{K}_\nabla}(m) = \mathsf{c}_{\mathsf{S}_A(M)}\left( \mathsf{T}(\mathsf{K}_\nabla)\left( \mathsf{K}_\nabla(m) \right) \right) - \mathsf{T}(\mathsf{K}_\nabla)\left( \mathsf{K}_\nabla(m) \right) \\
    =  \sum^n_{i=1} \sum^{n_i}_{j=1} m_{i,j} \mathsf{d}^\prime(a_{i,j})  \mathsf{d}\left( a_i \right) -m_{i,j} \mathsf{d}(a_{i,j})  \mathsf{d}^\prime\left( a_i \right)  
\end{gather*}
So we conclude that $\mathsf{C}_{\mathsf{K}_\nabla}(a) = a$ and $\mathsf{C}_{\mathsf{K}_\nabla}(m) = \psi\left( \nabla^2(m) \right)$ as desired. Then if $\nabla^2 =0$, we have that $\mathsf{C}_{\mathsf{K}_\nabla}(a) = a$ and $\mathsf{C}_{\mathsf{K}_\nabla}(m) = 0$, which says that $(\mathsf{K}_\nabla, \mathsf{H}_\nabla)$ is flat.  
\end{proof}

Conversely, expressing $\nabla^2$ in terms of $\mathsf{C}_{\mathsf{K}_\nabla}$ is usually only possible up to a factor of $2$. Indeed, ideally, one would like to have a retract of $\psi$. However, this is not possible in general. Instead, the best one can do in general is define the map $\phi: \mathsf{T}^2(\mathsf{S}_A(M)) \to \Omega^2(A) \otimes_A M$ which is defined as mapping monomials of the form $m \mathsf{d}(a) \mathsf{d}^\prime(b)$ to:
\[ \phi\left( m \mathsf{d}(a) \mathsf{d}^\prime(b) \right) =  \left(\mathsf{d}(a) \wedge \mathsf{d}(b) \right) \otimes_A m   \]
and mapping monomials of other forms to $0$, and then extend by linearity. 

\begin{lemma}\label{lemma:curvature2} The following equality holds: 
\[ 2 \nabla^2(m) = \phi\left( \mathsf{C}_{\mathsf{K}_\nabla} (m) \right) \]
\end{lemma}
\begin{proof} Recall that $\mathsf{d}(a) \wedge \mathsf{d}(b)  = - \mathsf{d}(b) \wedge \mathsf{d}(a)$. Then we get that: 
\begin{gather*}
(\mathsf{d}(a) \wedge \mathsf{d}(b)) \otimes_A m  \xmapsto{\psi} m \mathsf{d}(a)\mathsf{d}^\prime(b) - m \mathsf{d}^\prime(a)\mathsf{d}(b) \\
\xmapsto{\phi} \left(\mathsf{d}(a) \wedge \mathsf{d}(b) \right) \otimes_A m  - \left(\mathsf{d}(b) \wedge \mathsf{d}(a) \right) \otimes_A m  \\
= \left(\mathsf{d}(a) \wedge \mathsf{d}(b) \right) \otimes_A m  +\left(\mathsf{d}(a) \wedge \mathsf{d}(b) \right) \otimes_A m  \\ 
= 2 \left( \left(\mathsf{d}(a) \wedge \mathsf{d}(b) \right) \otimes_A m \right)
\end{gather*}
So $\phi(\psi(x)) = 2 x$. Now by Lemma \ref{lemma:curvature1}, $\mathsf{C}_{\mathsf{K}_\nabla}(m) = \psi\left( \nabla^2(m) \right)$. So post-composing both sides by $\phi$ gives us that $2 \nabla^2(m) = \phi\left( \mathsf{C}_{\mathsf{K}_\nabla} (m) \right)$. 
\end{proof}

As we see in the proof, the factor of $2$ comes from the fact that $\phi\left( m \mathsf{d}(a) \mathsf{d}^\prime(b) \right) = \left(\mathsf{d}(a) \wedge \mathsf{d}(b) \right) \otimes_A m = - \left(\mathsf{d}(b) \wedge \mathsf{d}(a) \right) \otimes_A m = - \phi\left( m \mathsf{d}^\prime(b) \mathsf{d}^\prime(a) \right)$. Thus essentially, this factor of 2 appears from the fact that we are going from a commutative setting to an anticommutative setting. If we are in a setting where $2$ is invertible, then the curvature of a module connection on $M$ is determined by the curvature of its associated connection on $\mathsf{q}_M$. Moreover, in such a setting, we also get that the curvature of a module connection on $M$ is zero if and only if its associated connection on $\mathsf{q}_M$ is flat. 

\begin{corollary}\label{cor:curvatures_match} If $2$ is a unit in $A$, then the following equality holds: 
\[ \nabla^2(m) = \frac{1}{2}\phi\left( \mathsf{C}_{\mathsf{K}_\nabla} (m) \right) \]
Furthermore, $\nabla^2 = 0$ if and only if $(\mathsf{K}_\nabla, \mathsf{H}_\nabla)$ is flat.
\end{corollary}

We now discuss the analogue of torsion in the tangent category sense in the algebraic framework. Recall that torsion was defined for affine tangent category connections, which we now know correspond to module connections on $\Omega(A)$. So here we introduce the natural notion of torsion for a connection of type $\nabla: \Omega(A) \to \Omega(A) \otimes_A \Omega(A)$. To the best of our knowledge, the notion of torsion of a module connection on the module of Kähler differentials has not been previously defined\footnote{Note: this is not the same as the usual notion of torsion of a module.}. 

\begin{definition}\label{def:module-torsion} The \textbf{torsion} of a connection $\nabla: \Omega(A) \to \Omega(A) \otimes_A \Omega(A)$ is the $R$-linear map $\widehat{\nabla}: \Omega(A) \to \Omega^2(A)$ defined as the following composite: 
    \begin{align}
\begin{array}[c]{c}
\widehat{\nabla} \end{array}  := \begin{array}[c]{c} \xymatrixrowsep{1pc}\xymatrixcolsep{2.5pc}\xymatrix{\Omega(A) \ar[r]^-{ \nabla} & \Omega(A) \otimes_A \Omega(A) \ar[r]^-{\omega} & \Omega^2(A)  }\end{array}
\end{align}
\end{definition}

On the other hand, given a module connection $\nabla: \Omega(A) \to \Omega(A) \otimes_A \Omega(A)$, the torsion of the induced affine connection $(\mathsf{K}_\nabla, \mathsf{H}_\nabla)$ is the $R$-algebra morphism $\mathsf{C}_{\mathsf{K}_\nabla}: \mathsf{T}(A) \to \mathsf{T}^2(A)$ defined as the dual of (\ref{eq:torsion-K}) or equivalently as in (\ref{eq:torsion-H}): 
\[ \mathsf{V}_{\mathsf{K}_\nabla} \colon = \mathsf{c}_{A} \circ \mathsf{K}_\nabla  -_{\mathsf{p}_A} \mathsf{K}_\nabla = \lbrace \mathsf{V}^\flat_{\mathsf{K}_\nabla} \rbrace \]
Then $(\mathsf{K}_\nabla, \mathsf{H}_\nabla)$ is torsion-free if its torsion is given on the generators by: 
\begin{align*}
   \mathsf{V}_{\mathsf{K}_\nabla}(a) = a &&    \mathsf{V}_{\mathsf{K}_\nabla}(\mathsf{d}(a)) = 0
\end{align*}

We will now explain how the torsion of an affine connection on $A$ is completely determined by the torsion of its associated connection on $\Omega(A)$. To do so, consider the map $\widehat{\psi}: \Omega^2(A) \to \mathsf{T}^2(A)$ defined as: 
\[ \widehat{\psi}(a \mathsf{d}(a) \wedge \mathsf{d}(c) ) = a \mathsf{d}(a)\mathsf{d}^\prime(c) - a \mathsf{d}^\prime(a)\mathsf{d}(c) \]
We will use this map to express $\mathsf{V}_{\mathsf{K}_\nabla}$ in terms of $\widehat{\nabla}$. As such, it also follows that if the torsion on a connection on $\Omega(A)$ is zero, then its associated affine connection on $A$ is torsion free. 

\begin{lemma}\label{lemma:torsion1} On generators, $\mathsf{V}_{\mathsf{K}_\nabla}$ satisfies the following: 
\begin{align*}
    \mathsf{V}_{\mathsf{K}_\nabla}(a) = a && \mathsf{V}_{\mathsf{K}_\nabla}(\mathsf{d}(a)) = \widehat{\psi}\left(  \widehat{\nabla}(\mathsf{d}(a))  \right) 
\end{align*}  
Moreover, if $\widehat{\nabla} =0$, then $(\mathsf{K}_\nabla, \mathsf{H}_\nabla)$ is torsion-free. 
\end{lemma}
\begin{proof} For the calculations of this proof, we set: 
\begin{align*}
    \nabla(\mathsf{d}(a)) = \sum^n_{i=1} a_i \mathsf{d}(b_i) \otimes_A \mathsf{d}(c_i) 
\end{align*}
As such: 
\begin{gather*}
   \widehat{\nabla}(\mathsf{d}(a)) = \sum^n_{i=1} a_i \left( \mathsf{d}(b_i) \wedge \mathsf{d}(c_i) \right)
\end{gather*}
and so: 
\begin{gather*}
   \omega\left( \widehat{\nabla}(\mathsf{d}(a)) \right) = \sum^n_{i=1} a_i\mathsf{d}(b_i) \mathsf{d}^\prime(c_i) - \sum^n_{i=1} a_i\mathsf{d}^\prime(b_i) \mathsf{d}(c_i)
\end{gather*}
On the other hand, it will be easier to work out the torsion of the induced affine connection using the horizontal connection formula. So by the bracketing operation, the torsion is given on generators as:
\begin{gather*}
    \mathsf{V}_{\mathsf{K}_\nabla}(a) = \lbrace \mathsf{V}^\flat_{\mathsf{K}_\nabla} \rbrace(a) = \mathsf{V}^\flat_{\mathsf{K}_\nabla} (a) \\  
    \mathsf{V}_{\mathsf{K}_\nabla}(\mathsf{d}(a)) = \lbrace \mathsf{V}^\flat_{\mathsf{K}_\nabla} \rbrace(\mathsf{d}(a)) = \mathsf{V}^\flat_{\mathsf{K}_\nabla} (\mathsf{d}^\prime\mathsf{d}(a)) 
\end{gather*}
So let's work out $\mathsf{V}_{\mathsf{K}_\nabla}$ on these generators. First note that by definition, we have that the horizontal connection on the generators $a$ and $\mathsf{d}^\prime\mathsf{d}(a)$ gives:  
\begin{align*}
    \mathsf{H}_\nabla(a) = a \otimes_A 1 &&     \mathsf{H}_\nabla(\mathsf{d}^\prime\mathsf{d}(a)) = \nabla( \mathsf{d}(a) ) =   \sum^n_{i=1} a_i \mathsf{d}(b_i) \otimes_A \mathsf{d}(c_i) 
\end{align*}
while $\mathsf{U}_{\mathsf{p}_A}$ on these generators is: 
\begin{align*}
   \mathsf{U}_{\mathsf{p}_A}(a \otimes_A b) = ab &&    \mathsf{U}_{\mathsf{p}_A}(\mathsf{d}(a) \otimes_A \mathsf{d}(b)) = \mathsf{d}^\prime(a) \mathsf{d}(b)
\end{align*}
So, the composite on these generators gives us: 
\begin{align*}
   \mathsf{U}_{\mathsf{p}_A}\left( \mathsf{H}_\nabla(a) \right) = a &&    \mathsf{U}_{\mathsf{p}_A}\left( \mathsf{H}_\nabla(\mathsf{d}^\prime\mathsf{d}(a)) \right) = \sum^n_{i=1} a_i \mathsf{d}^\prime(b_i)\mathsf{d}(c_i) 
\end{align*}
Now recall that the canonical flip on the generators $a$ and $\mathsf{d}^\prime\mathsf{d}(a)$ does nothing, so we get that: 
\begin{align*}
   \mathsf{U}_{\mathsf{p}_A}\left( \mathsf{H}_\nabla(\mathsf{c}_A(a)) \right) = a &&    \mathsf{U}_{\mathsf{p}_A}\left( \mathsf{H}_\nabla(\mathsf{c}_A \left(\mathsf{d}^\prime\mathsf{d}(a) \right) ) \right) = \sum^n_{i=1} a_i \mathsf{d}^\prime(b_i)\mathsf{d}(c_i) 
\end{align*}
On the other hand, the canonical flip swaps $\mathsf{d}$ and $\mathsf{d}^\prime$, so we also get that: 
\begin{align*}
  \mathsf{c}_A\left( \mathsf{U}_{\mathsf{p}_A}\left( \mathsf{H}_\nabla(a) \right) \right) = a &&     \mathsf{c}_A \left(  \mathsf{U}_{\mathsf{p}_A}\left( \mathsf{H}_\nabla(\mathsf{d}^\prime\mathsf{d}(a)) \right) \right) = \sum^n_{i=1} a_i \mathsf{d}(b_i)\mathsf{d}^\prime(c_i) 
\end{align*}
So their subtraction (in the differential bundle sense) on these generators is: 
\begin{gather*}
\mathsf{V}^\flat_{\mathsf{K}_\nabla}(a) = a \\ 
   \mathsf{V}^\flat_{\mathsf{K}_\nabla} (\mathsf{d}^\prime\mathsf{d}(a))  = \sum^n_{i=1} a_i\mathsf{d}(b_i) \mathsf{d}^\prime(c_i) - \sum^n_{i=1} a_i\mathsf{d}^\prime(b_i) \mathsf{d}(c_i)
\end{gather*}
So we conclude that $ \mathsf{V}_{\mathsf{K}_\nabla}(a) = a$ and $\mathsf{V}_{\mathsf{K}_\nabla}(\mathsf{d}(a)) = \widehat{\psi}\left(  \widehat{\nabla}(\mathsf{d}(a))  \right)$ as desired. Now if $\widehat{\nabla} =0$, then $\mathsf{V}_{\mathsf{K}_\nabla}(a) = a$ and $\mathsf{V}_{\mathsf{K}_\nabla}(\mathsf{d}(a)) = 0$, so we get that $(\mathsf{K}_\nabla, \mathsf{H}_\nabla)$ is torsion-free. 
\end{proof}

Conversely, similarly to the situation for curvature, expressing $\widehat{\nabla}$ in terms of $\mathsf{V}_{\mathsf{K}_\nabla}$ is usually only possible up to a factor of $2$. So define the map $\widehat{\phi}: \mathsf{T}^2(A) \to \Omega^2(A)$ on monomials of the form $a \mathsf{d}(b) \mathsf{d}^\prime(c)$ as follows: 
\[ \widehat{\phi} (a \mathsf{d}(b) \mathsf{d}^\prime(c)) = a \mathsf{d}(a) \wedge \mathsf{d}(c) \]
and mapping monomials of other forms to $0$, and then extend by linearity. 

\begin{lemma} The following equality holds: 
\[ 2 \widehat{\nabla}(\mathsf{d}(a)) = \widehat{\phi}\left( \mathsf{V}_{\mathsf{K}_\nabla}(\mathsf{d}(a)) \right) \]
\end{lemma}
\begin{proof} By similar arguments as in the proof of Lemma \ref{lemma:curvature2}, we get that $\widehat{\phi}(\widehat{\psi}(x)) = 2 x$. Now by Lemma \ref{lemma:torsion1}, $\mathsf{V}_{\mathsf{K}_\nabla}(\mathsf{d}(a)) = \widehat{\psi}\left(  \widehat{\nabla}(\mathsf{d}(a))  \right)$. So post-composing both sides by $\widehat{\phi}$ gives us that $2 \widehat{\nabla}(\mathsf{d}(a)) = \widehat{\phi}\left( \mathsf{V}_{\mathsf{K}_\nabla}(\mathsf{d}(a)) \right)$. 
\end{proof}

As such, it follows that if we are in a setting where $2$ is invertible, then the torsion of a connection on $\Omega(A)$ is determined by the torsion of its associated affine connection on $A$. Moreover, in such a setting, we also get that the torsion of a module connection on $\Omega(A)$ is zero if and only if its associated affine connection on $A$ is torsion-free. 

\begin{corollary}\label{cor:torsion-2} If $2$ is a unit in $A$, then the following equality holds: 
\[ \widehat{\nabla}(m) = \frac{1}{2}\widehat{\phi}\left( \mathsf{V}_{\mathsf{K}_\nabla}(\mathsf{d}(a)) \right) \]
Furthermore, $\widehat{\nabla} = 0$ if and only if $(\mathsf{K}_\nabla, \mathsf{H}_\nabla)$ is torsion-free.
\end{corollary}

\subsection{Curvature Examples} \label{sec:sexamples-curvature}

We now consider the curvature of some of the examples from Sec \ref{sec:examples}.  

\begin{example}
Recall from Ex \ref{ex:connections_on_plane} that the canonical module connection on the Kähler differentials of the $R$-affine plane $A = R[x,y]$ is defined by
    \[ \nabla(\mathsf{d}(x_1)) = \nabla(\mathsf{d}(x_2)) = 0 \]
 We claim that this module connection is flat; that is, that $\widehat{\nabla}$ is constantly zero. Indeed, consider what $\widehat{\nabla}$ does to a term of the form $p(x,y)\mathsf{d}(x_i)$ (for $i \in \{1,2\}$, and where $p(x,y)$ is an arbitrary element of $A$).  After applying $\nabla$, one gets
    \[ \left( \frac{\mathsf{d}p(x,y)}{\mathsf{d}(x_1)} \mathsf{d}(x_1) + \frac{\mathsf{d}p(x,y)}{\mathsf{d}(x_2)}\mathsf{d}(x_2) \right) \otimes \mathsf{d}(x_i) \]
Then applying $1 \otimes \nabla$ and then $\omega \otimes 1$ gives
  { \small{ \[ \left( \frac{\mathsf{d}^2p(x,y)}{\mathsf{d}(x_1) \mathsf{d}(x_1)}\mathsf{d}(x_1) \wedge \mathsf{d}(x_1) + \frac{\mathsf{d}^2p(x,y)}{\mathsf{d}(x_1) \mathsf{d}(x_2)} \mathsf{d}(x_1) \wedge \mathsf{d}(x_2) + \frac{\mathsf{d}^2p(x,y)}{\mathsf{d}(x_2) \mathsf{d}(x_1)}\mathsf{d}(x_2) \wedge \mathsf{d}(x_1) + \frac{\mathsf{d}^2p(x,y)}{\mathsf{d}(x_2) \mathsf{d}(x_2)}\mathsf{d}(x_2) \wedge \mathsf{d}(x_2) \right) \otimes \mathsf{d}(x_i) \]} } %
which is indeed equal to $0$ by antisymmetry of $\wedge$ and symmetry of mixed partial derivatives.  
\end{example}

\begin{example}
There exist non-flat connections on the module of Kähler differentials of the affine plane; for example, it is easy to see that:
\begin{align*}
    \nabla(\mathsf{d}(x_1)) := x_2\mathsf{d}(x_1) \otimes \mathsf{d}(x_1) && \nabla(\mathsf{d}(x_2)) := 0
\end{align*}
is not flat (for example, by checking what $\widehat{\nabla}$ does to $\mathsf{d}(x_1)$; we leave this to the reader).  
\end{example}

\begin{example}
For a different type of example, consider again the affine circle
    \[ A := R[x_1, x_2]/(x_1^2 + x_2^2 - 1) \]
for a commutative ring $R$ of characteristic not equal to $2$. We claim that all module connections on the module of Kähler differentials of $A$ are flat. To prove this, it suffices to show that $\Omega^2(A)$ is the zero module. For this, consider a term of the form (for simplicity, we also drop the subscript of the $\wedge$ product of elements): 
    \[ \mathsf{d}(x_1) \wedge \mathsf{d}(x_2) \]
In $A$, $x_1^2 + x_2^2 = 1$, so we can rewrite the above as: 
\begin{align} \label{Geoff-star}
    x_1^2 (\mathsf{d}(x_1) \wedge \mathsf{d}(x_2)) + x_2^2 (\mathsf{d}(x_1) \wedge \mathsf{d}(x_2))  
\end{align}
But recall that in $\Omega(A)$, we have that: 
    \[ 2x_1\mathsf{d}(x_1) + 2x_2\mathsf{d}(x_2), \]
Since the characteristic of $R$ is not equal to $2$, we then get that: 
    \[ x_1\mathsf{d}(x_1) = -x_2\mathsf{d}(x_2) \]
As such, we can rewrite (\ref{Geoff-star}) as: 
    \[ -x_1x_2(\mathsf{d}(x_1) \wedge \mathsf{d}(x_1)) - x_1x_2 (\mathsf{d}(x_2) \wedge \mathsf{d}(x_2)) \]
which is equal to $0$. Thus, all terms in $\Omega^2(A)$ are equal to $0$.  Thus for any module connection $\nabla$ on $\Omega(A)$, the codomain of its curvature $\Omega^2(A)$ is the $0$ module, and hence all such connections are flat. We stress that this is particular to the circle; it does not apply to the higher dimensional spheres (Ex \ref{ex:higher_spheres}). 
\end{example}

\subsection{Tangent Category Connections for Schemes}

The category of schemes is a Rosický tangent category \cite[Prop 4.6]{cruttwell2023differential}, and the differential bundles over a scheme correspond to quasi-coherent sheaves \cite[Thm 4.28]{cruttwell2023differential}.  The results of Sec \ref{sec:tan_cat_connections_in_aff} for affine schemes extend immediately to connections in this more general setting; we invite the reader to see \cite[Def 1.0] {katz1970nilpotent} for the definition of connections on quasi-coherent sheaves.   

\begin{corollary}\label{cor:scheme_connections}
If $B$ is a scheme, $X$ a scheme over $B$, and $E$ a quasi-coherent sheaf on it, then the data for a $B$-connection on $E$ is equivalent to the data of a tangent category connection on the associated differential bundle over $X$ in the Rosický tangent category of schemes over $B$.
\end{corollary}
\begin{proof}
This follows from Thm \ref{thm:affine_connection_correspondence} and the fact that everything is defined affine-locally.  
\end{proof}

A particularly useful example to illustrate this idea is the lack of connections on the sheaf of Kähler differentials of the projective line.  It is worth first noting that this fact can be seen as a corollary of more general results, as follows:

\begin{enumerate}[(i)]
    \item The Kähler module of the projective line $P_1$ is the line bundle ${\mathcal O}_{P_1}(-2)$ \cite[21.4.2]{vakil}.
    \item Over a field of characteristic zero, a line bundle $L$ admits a connection if and only if its first Chern class vanishes\footnote{This seems to be a well-known folklore result; for example, see \url{https://mathoverflow.net/questions/123942/}}.
    \item For any integer $m \ne 0$, the first Chern class of ${\mathcal O}_{P_1}(m)$ is non-vanishing \cite[Ex 11]{computing_chern}.  
\end{enumerate} 
Combining these results together tells us that there are no connections on the module of Kähler differentials, and thus no tangent category connection on the tangent bundle of the projective line either. However, it will also be useful to calculate this explicitly, as it illustrates how connections on each affine patch must interact with one another.  

\begin{example}\label{ex:no_connection_projective_line}
Recall that the projective line over a commtuative ring $R$ is given by gluing two copies of the affine line $A_1 = R[x]$ and $A_2 = R[y]$ along their open sets $R[x]_x$ (i.e. $R[x]$ localized at the multiplicative set generated by $x$) and $R[y]_y$ via the isomorphism $R[x]_x \to R[y]_y$ given by:
\[x \mapsto \frac{1}{y}\] 
This induces an isomorphism between the module of Kähler differentials $\Omega(R[x]_x) \to \Omega(R[y]_y)$ given by:
    \[ \mathsf{d}(x) \mapsto \mathsf{d} \left( \frac{1}{y} \right) = \frac{-1}{y^2} \mathsf{d}(y) \]
Then, giving a connection on this scheme is equivalent to giving a connection on each copy of the affine line, which is compatible with this gluing. That is, one must give module connections $\nabla_1$ on $R[x]$ and $\nabla_2$ on $R[y]$ such that the following diagram commutes
\[
\xymatrix{\Omega(R[x]_x) \ar[rr]^-{\nabla_1} \ar[d] & & \Omega(R[x]_x) \otimes \Omega(R[x]_x) \ar[d] \\ \Omega(R[y]_y) \ar[rr]_-{\nabla_2} & & \Omega(R[y]_y) \otimes \Omega(R[y]_y)}
\]
where the arrows going down are the transition isomorphisms described above. By Ex \ref{ex:connections_on_affine_n_space}, $\nabla_1$ and $\nabla_2$ must be of the form
\begin{align*}
    \nabla_1(\mathsf{d}(x)) = p(x)\mathsf{d}(x) \otimes \mathsf{d}(x) && \nabla_2(\mathsf{d}(y)) = q(y) \mathsf{d}(y) \otimes \mathsf{d}(y)
\end{align*}
Now consider starting with $\mathsf{d}(x)$ in the top-left of the above diagram.  Going right then down gives
    \[ \mathsf{d}(x) \mapsto p(x) \mathsf{d}(x) \otimes \mathsf{d}(x) \mapsto p\left (\frac{1}{y} \right) \left[ \frac{-1}{y^2}\mathsf{d}(y) \otimes \frac{-1}{y^2}\mathsf{d}(y) \right]  = \frac{1}{y^4} p\left( \frac{1}{y} \right) \mathsf{d}(y) \otimes \mathsf{d}(y) \]
while going down then right gives (where the last equality is by the Leibniz rule for $\nabla_2$): 
    \[ \mathsf{d}(x) \mapsto \frac{-1}{y^2} \mathsf{d}(y) \mapsto \frac{2}{y^3} \mathsf{d}(y) \otimes \mathsf{d}(y) - \frac{1}{y^2} q(y) \mathsf{d}(y) \otimes \mathsf{d}(y) \]
Putting these together means one would need to find polynomials $p(x)$ and $q(y)$ so that
    \[ \frac{2}{y^3} = \frac{1}{y^2}q(y) + \frac{1}{y^4}p \left( \frac{1}{y} \right) \]
But unless the characteristic of $R$ is $2$, the expression on the left is a term of degree $-3$, while the first expression on the right is of degree $\geq -2$ and the second expression of degree $\leq -4$. Thus, unless the characteristic of $R$ is $2$, it is impossible to find such $p$ and $q$, and hence, no connection can exist on the projective line in this case. If $R$ is of characteristic $2$, then $p(x) = q(y) = 0$ is a connection on the projective line over $R$ (in fact, the only one).  
\end{example}

\section{Applying Tangent Category Theory}\label{sec:applying_theory}

In this section, we discuss how general results about connections in a tangent category apply to the particular case of the tangent category of affine schemes.  In many cases, these results recreate previously known results in algebraic geometry, but in some cases, we get new results and/or highlight results that are not that prominent (but perhaps should be).

We begin by briefly recalling four essential results about the existence of connections in an arbitrary Rosický tangent category. 
\begin{enumerate}[(i)]
    \item \textbf{Differential objects have unique connections}: In a \emph{Cartesian} Rosický tangent category \cite[Def 2.8]{cockett2014differential} (i.e. a tangent category with finite products that are preserved up to isomorphism by the tangent bundle functor), a \textbf{differential object} \cite[Prop 3.4]{cockett2018differential} is a differential bundle over the terminal object. Every such differential bundle has a \emph{unique} connection on it \cite[Prop. 5.3]{connections}. 
    \item \textbf{Connections are preserved by pullback}: If $\mathsf{q}: E \to A$ is a differential bundle with a connection and $f: X \to A$ is any map for which the pullback $f^\ast(\mathsf{q})$ of $q$ along $f$ exists and is preserved by $\mathsf{T}$, then $f^\ast(\mathsf{q}) \to X$ is a differential bundle and inherits a connection from $q$ \cite[Prop. 5.6]{connections}.
    \item \textbf{Connections are preserved by application of $\mathsf{T}$}: If $\mathsf{q}: E \to A$ is a differential bundle with a connection, then $\mathsf{T}(\mathsf{q}): \mathsf{T}(E) \to \mathsf{T}(A)$ inherits a connection from $\mathsf{q}$ \cite[Prop. 5.5]{connections}.
    \item \textbf{Connections are preserved by retract}: If $\mathsf{q}: E \to A$ is a differential bundle with a connection, $\mathsf{q}^\prime: E^\prime \to A^\prime$ is a differential bundle, and we have a section/retraction pair $\mathsf{q}^\prime \to \mathsf{q}$ and ${\mathsf{q} \to \mathsf{q}}$ of differential bundle morphisms, then $\mathsf{q}^\prime$ inherits a connection\footnote{We note that in \cite[Prop. 4.12]{connections}, the statement is for a tangent category, and thus only says that one obtains a \emph{horizontal} connection on $\mathsf{q}^\prime$. Indeed, as discussed after \cite[Ex 5.7]{connections}, a retract of a full connection is not necessarily a full connection. However, since in this paper, we are working in a Rosický tangent category, we do indeed obtain a full connection.} from $\mathsf{q}$ \cite[Prop. 4.12]{connections}.
\end{enumerate}

By the results of this paper, we then get four corresponding results about connections on modules/quasi-coherent sheaves of modules on affine schemes/schemes. 

\begin{proposition}\label{prop:schmes:1-4} Let $A$ be a scheme. Then:
\begin{enumerate}[(i)]
    \item \label{prop:schmes:1} If $M$ is a quasi-coherent sheaf of modules on $A$, then $M$ has a natural choice of $A$-connection.
    \item \label{prop:schmes:2} Suppose $B$ and $X$ are $A$-schemes, $M$ is a quasi-coherent sheaf of modules on $B$, that $M$ has an $A$-connection $\nabla$, and we have a scheme morphism $f: X \to B$.  Then the corresponding inverse image/pullback $X$-sheaf of modules $f^\ast(M)$ also has an $A$-connection.
    \item \label{prop:schmes:3} Suppose $B$ is an $A$-scheme, $M$ is a quasi-coherent sheaf of modules on $B$, and $M$ has an $A$-connection $\nabla$. Then the $\mathsf{T}(A)$-sheaf $\Omega_{\mathsf{T}(A)}(\mathsf{Spec}(\mathsf{Sym}(M))$ has an $A$-connection.  
    \item \label{prop:schmes:4} Suppose $B$ is an $A$-scheme, $M$ and $M^\prime$ are quasi-coherent sheafs of modules on $B$ with a pair of section/retraction maps $s: M^\prime \to M$ and $r: M \to M^\prime$, and $M$ has an $A$-connection $\nabla$.  Then $M^\prime$ also has an $A$-connection.   
\end{enumerate}
\end{proposition}

It is worth commenting on the importance of these results. 
\begin{enumerate}[(i)]
    \item This is immediate without any recourse to tangent category theory. For example, in the affine case of a module $M$ over a base commutative ring $R$, an $R$-connection would consist of a map  
    \[ \nabla: M \to \Omega(k) \otimes_R M \]
    but $\Omega(k) \cong \mathsf{0}$, so this must simply be the constantly zero map.
    \item This seems to be a well-known result in algebraic geometry, though we can find no precise reference in papers on connections.  The particular case of flat connections pulling back can be found in any textbook on $D$-module theory (more discussion on this below).     
    \item We cannot find any reference to a result like this.
    \item Once again, this result does not seem to be that well-known, though it is not difficult to prove directly. In the affine case, given modules $M$ and $M^\prime$ with a section/retraction pair $s: M^\prime \to M$ and ${r: M \to M^\prime}$, and a connection $\nabla$ on $M^\prime$, we get a connection $\nabla^\prime$ on $M^\prime$ by the composite:
\[ \nabla^\prime := \xymatrixrowsep{1pc}\xymatrixcolsep{5pc}\xymatrix{ M^\prime \ar[r]^-{s} & M \ar[r]^-{\nabla} & \Omega(A) \otimes_A M \ar[r]^-{1 \otimes r} & \Omega(A) \otimes_A M^\prime} \] 
\end{enumerate}

In addition, combinations of these results also lead to more interesting conclusions.  Combining (\ref{prop:schmes:1}) and (\ref{prop:schmes:2}), we get the following:  

\begin{lemma}\label{lemma:pullback_connections}
Let $A$ be a scheme. If $M$ is a quasi-coherent sheaf of modules on $A$ and $B$ is an $A$-scheme, then the pullback $B$-quasi-coherent sheaf $B \otimes_A M$ has a canonical choice of connection $\nabla$ given locally by
    \[ \nabla(a \otimes m) = \mathsf{d}(a) \otimes (1 \otimes m) \]
\end{lemma}
\begin{proof}
This is simply a result of combining (\ref{prop:schmes:1}) and (\ref{prop:schmes:2}) from Prop \ref{prop:schmes:1-4}; the particular form of the connection is found by applying the pullback connection construction to the canonical connection on $M$.  
\end{proof}

In the affine case, when $M$ is a free $A$-module on $n$ generators, we have $B \otimes_A M \cong B^n$, and the corresponding connection is precisely the one described in Ex \ref{ex:conn_an}, which recall was induced by the map given by: 
\[ (b_1, b_2, \hdots,, b_n) \mapsto (\mathsf{d}(b_1), \mathsf{d}(b_2), \hdots, \mathsf{d}(b_n)) \]

Then combined with (\ref{prop:schmes:4}), we then get the following:

\begin{lemma}
If $A$ is a commutative $R$-algebra, $B$ an $A$-algebra, and $M$ a finitely generated projective $B$-module, then $M$ has an $A$-connection.
\end{lemma}
\begin{proof}
If $M$ is finitely generated and projective, then it is a retract of a finite free module, so by (\ref{prop:schmes:4}), and the result above about connections on free modules, $M$ acquires a connection.  
\end{proof}

This particular result appears in \cite[Prop. 2.5]{rojas_mendoza}.  However, it's worth noting that it also implies the following, which we have not found a reference for (although it is surely also well-known):

\begin{corollary}\label{cor:smooth_have_connections}
If $B$ is a smooth $A$-algebra (see \cite[Sec 10.137.1]{stacks-project} for the definition of a smooth algebra), then its module of Kähler differentials over $A$, $\Omega_A(B)$, has a connection.  
\end{corollary}
\begin{proof}
By \cite[Sec 10.142.(2)]{stacks-project}, if $B$ is smooth then $\Omega_A(B)$ is finitely generated projective, so by the previous result, acquires a connection.  
\end{proof}

As mentioned earlier, this seems to be a non-obvious result.  That is, if you are given a particular smooth affine variety, it's not clear how to define a connection on its module of Kähler differentials.  However, the above not only shows that it is always possible to do so, but also gives an explicit way to find such a connection. To illustrate this idea with a particular example, we return to the affine circle. In Ex \ref{ex:affine_circle_connection}, we gave a connection on its module of Kähler differentials; here, we show how to view this as a retract of the canonical connection on a free module.  

\begin{example}\label{ex:affine_circle_revisited} Let $R$ be a commutative ring of characteristic not equal to 2, and let $B$ be the affine circle 
    \[ R[x,y]/\<x^2 + y^2 -1\>. \]
Then recall that its module of Kähler differentials can be described as: 
    \[ \Omega(B) = B\mathsf{d}(x) \oplus B\mathsf{d}(y)/\<2x\mathsf{d}(x) + 2y\mathsf{d}(y)\> \]
Moreover, $\Omega(B)$ is a retract of the free $B$-module $B\mathsf{d}(x) \oplus B\mathsf{d}(y)$; the retraction $r$ is the quotient map, and the section 
    \[ s: B\mathsf{d}(x) \oplus B\mathsf{d}(y)/\<2x\mathsf{d}(x) + 2y\mathsf{d}(y)\> \to B\mathsf{d}(x) \oplus B\mathsf{d}(y) \]
is given on generators by: 
\begin{align*}
    \mathsf{d}(x) \mapsto y^2\mathsf{d}(x) - xy\mathsf{d}(y) && \mathsf{d}(y) \mapsto -xy\mathsf{d}(x) + x^2\mathsf{d}(y)
\end{align*}
As above, there is a canonical connection $\nabla$ on $B\mathsf{d}(x) \oplus B\mathsf{d}(y)$ which sends both $\mathsf{d}(x)$ and $\mathsf{d}(y)$ to $0$.  Then as above, the induced connection $\nabla'$ on $\Omega(B)$ is given by the composite
    \[ \Omega(B) \to^{s} B\mathsf{d}(x) \oplus B\mathsf{d}(y) \to^{\nabla} \Omega_{B} \otimes (B\mathsf{d}(x) \oplus B\mathsf{d}(y)) \to^{1 \otimes r} \Omega(B) \otimes \Omega(B) \]
Starting with $\mathsf{d}(x)$, applying $s$ we get: 
    \[ y^2\mathsf{d}(x) - xy \mathsf{d}(y) \]
then applying $\nabla$ we get: 
    \[ d(y^2) \otimes \mathsf{d}(x) + y^2 \nabla(\mathsf{d}(x)) - d(xy) \otimes \mathsf{d}(y) - xy \nabla(\mathsf{d}(y))  = 2y\mathsf{d}(y) \otimes \mathsf{d}(x) - x\mathsf{d}(y) \otimes \mathsf{d}(y) - y \mathsf{d}(x) \otimes \mathsf{d}(y)\]
Then applying the quotient map $r$, this is an element of $\Omega(B)$; in this module $x\mathsf{d}(x) = -y\mathsf{d}(y)$, and hence the above reduces to
    \[ -x\mathsf{d}(x) \otimes \mathsf{d}(x) - x \mathsf{d}(y) \otimes \mathsf{d}(y) \]
which is exactly what the connection on the affine circle described in Ex \ref{ex:affine_circle_connection} does to $\mathsf{d}(x)$ (and one can similarly calculate $\mathsf{d}(y)$). 
\end{example}

It is also worth highlighting that Cor \ref{cor:smooth_have_connections} is \emph{not} true for smooth (non-affine) schemes; see Ex \ref{ex:no_connection_projective_line}.  

The last point of theory we'd like to discuss has to do with constructing new tangent categories.  In \cite[Thm 5.16]{affine_tan_cats}, it is shown that for any tangent category, one can make new tangent categories of objects equipped with a connection on their tangent bundle. There is also a full subcategory consisting of the objects whose connections are flat. Though we will not go into details here, one can similarly construct tangent categories whose objects are differential bundles with a connection over a fixed base (and a subcategory of objects whose connections are flat). This then shows that over a fixed base (affine) scheme $A$, one can construct a tangent category whose objects are quasi-coherent sheaves of modules with a (flat) connection. This is particularly relevant for the study of $D$-modules, as these are nothing more than modules over some fixed base equipped with a flat connection \cite[Lemma 1.2.1]{d-modules}.  Thus, the category of $D$-modules can be given the structure of a tangent category.   We hope to explore this point of view in future work.   

\section{Conclusion}\label{sec:conclusion}

The main results of this paper are as follows: 
\begin{enumerate}[(i)]
\item In the tangent category of commutative algebras, there are no non-trivial connections (Proposition \ref{prop:no_conn_in_algebra}).
\item In the tangent category of affine schemes, connections on differential bundles exactly correspond to connections on modules (Theorem \ref{thm:affine_connection_correspondence}).
\item In the tangent category of schemes, connections on differential bundles exactly correspond to connections on quasi-coherent sheaves of modules (Corollary \ref{cor:scheme_connections}).
\item Up to a factor of 2, the tangent category notion of curvature (torsion) and the module definition of curvature (torsion) correspond (Corollary \ref{cor:curvatures_match} and Corollary \ref{cor:torsion-2}).
\end{enumerate}

We conclude this paper by discussing some interesting future research projects that build upon the results and observations of this paper that we hope to pursue. 

\begin{itemize}
    \item As mentioned at the end of the previous section, $D$-modules can be given the structure of a tangent category.  What do various notions from tangent category theory look like when applied to this category?  Does this recreate existing ideas in $D$-module theory and/or tell us anything new about these objects?
    \item Recent work has shown that for any operad, the opposite of its category of algebras is a tangent category \cite{operad}.  This recreates the tangent category of affine schemes when applied to the commutative and unital operad. In addition, just as here, a differential bundle over an algebra $A$ is equivalent to an $A$-module \cite{diff_bundles_operad}.  This leads to the question of characterizing connections in this generality; in some cases, we expect this should recreate notions of connections in non-commutative geometry.  
    \item As first discussed by Ehresmann \cite{ehresmann:connections}, in differential geometry, connections can be considered on any submersion (that is, not just on vector bundles). In future work, we plan to develop a similar, more general notion of connection in any tangent category.  When applied to the tangent categories of affine schemes and/or schemes, it would then be interesting to compare it to the more general notion of connection in algebraic geometry due to Grothendieck (which generalizes the notion of connections on modules we have considered in this paper).  
\end{itemize}

\bibliographystyle{plain}      
\bibliography{connections-ref}   

\end{document}